\documentclass[12pt]{amsart} %\openup 10pt
\usepackage[text={6.5in,9in},centering]{geometry}

\usepackage{color}
\usepackage{amssymb,amsmath}
\usepackage{amscd}
\usepackage{graphicx}
\usepackage[config,labelfont={rm,rm},textfont=rm,aboveskip=0in,belowskip=2ex]{caption,subfig}
\DeclareMathAlphabet\mathbold{OML}{cmm}{b}{it} 
\graphicspath{{FigPDF/}{FigEPS/}}

\newtheorem{theorem}{Theorem}[section]
\newtheorem{proposition}[theorem]{Proposition}
\newtheorem{lemma}[theorem]{Lemma}
\newtheorem{remark}[theorem]{Remark}
\usepackage{bm}
\usepackage{xy}
\numberwithin{equation}{section} \numberwithin{table}{section} \numberwithin{figure}{section}

\newcommand{\Tri}[1]{\mathcal{T}_{#1}}

  \author{Guozhu Yu}
   \address{School of Mathematics, Sichuan University, Chengdu, Sichuan, China, 610064}
   \email{yuguozhumail@yahoo.com.cn}

    \author{Jinchao Xu}

  \address{Department of Mathematics, Pennsylvania State University,
    University Park, State College PA 16802}

   \email{xu@math.psu.edu}

   \author{Ludmil T. Zikatanov}
  \address{Department of Mathematics, Pennsylvania State University,
    University Park, State College PA 16802}

   \email{ltz@math.psu.edu}

   \title[two-level method for anisotropic diffusion equations]{Analysis of two-level method for anisotropic diffusion equations on
     aligned and non-aligned grids}

\thanks{This work was completed while the first author was visiting
the Center for Computational Mathematics and Applications (CCMA) at Penn
State. Thanks go to the CCMA at PSU for the support of  this work.  
The work of the third author is supported in part by the National Science
Foundation, DMS-0810982, and OCI-0749202}

\date{\today}

\begin{document}
\setlength{\abovecaptionskip}{0pt} \setlength{\belowcaptionskip}{0pt}
\begin{abstract}
  This paper is devoted to the multigrid convergence analysis for the
  linear systems arising from the conforming linear finite element
  discretization of the second order elliptic equations with
  anisotropic diffusion.  The multigrid convergence behavior is known
  to strongly depend on whether the discretization grid is aligned or
  non-aligned with the anisotropic direction and analyses in the paper
  will be mainly focused on two-level algorithms.  For an aligned grid
  case, a lower bound is given for point-wise smoother which shows
  deterioration of convergence rate.  In both aligned and non-aligned
  cases we show that for a specially designed block smoother the
  convergence is uniform with respect to both anisotropy ratio and mesh size
  in the energy norm. The analysis is complemented with numerical
  experiments which confirm the theoretical results.
\end{abstract}

% AMS subject classifications (used in AMS journals)

   % AMS keywords (used in AMS journals)

\maketitle

\section{Introduction}
In this paper we will study multilevel methods for anisotropic partial differential equations (PDEs) by finite element (FE) methods and in
particular we will analyze the convergence behavior of some methods for anisotropic diffusion equations on grids that are either aligned or
non-aligned with the anisotropy.

There are already many convergence results in the literature for multilevel method on anisotropic problems when the underlying FE grid is
aligned with the anisotropy direction.  The case of constant anisotropy was considered by Stevenson \cite{Stevenson1993,
  Stevenson1994}, who established a uniform convergence of the V-cycle
multigrid methods. Main tools in his analysis are the so called classical \emph{smoothing} and \emph{approximation} properties (see
Hackbusch~\cite{Hackbusch1985}). The case of ``mildly'' varying anisotropy was analyzed in a work by Bramble and Zhang~\cite{Bramble-Zhang2001}.
Using a different theoretical framework developed in Bramble, Pasciak, Wang and Xu~\cite{Bramble-Pasciak-Wang-Xu1991} and Xu~\cite{Xu1992},
Neuss~\cite{Neuss1998} also showed uniform convergence of the V-cycle algorithm for anisotropic diffusion problem. More recently, Wu, Chen, Xie
and Xu~ \cite{Wu-Chen-Xie-Xu} analyzed V-cycle multigrid with line smoother and standard coarsening, and V-cycle multigrid with point
Gauss-Seidel smoother and semi-coarsening at the same time, and they were able to prove convergence under weaker assumptions on the regularity
of the solution of the underlying PDE.  Another technique, based on tensor product type subspace splittings and a semi-coarsening was proposed
and analyzed by Griebel and Oswald~\cite{Griebel-Oswald1995}. They have shown uniform and optimal condition number bounds for multilevel
additive preconditioners.

The aforementioned theoretical convergence results on multigrid methods for the anisotropic diffusion equations are, however, all carried out
under one main assumption that the anisotropy direction is aligned with the mesh.  But such an assumption is not always satisfied in practice.
In this paper, we will make an attempt to develop uniform convergence theory in certain cases when this aligned grid assumption is not
satisfied.  More specifically, we will study the problem \eqref{equation:model1} in the case that $\Omega$ is a square domain triangulated by
uniform grids that is rotated by an angle $\omega\in[0, \pi]$.  We have grids that are not aligned with anisotropy except for special cases that
$\omega=0$, $\omega=\frac{\pi}{2}$ or $\omega=\frac{3\pi}{4}$.

For this special class of domains and grids, we will design a
two-level method and prove its uniform convergence (with respect to
both anisotropy and mesh size).  We are not yet able to extend our
theoretical analysis neither to multilevel (more than two levels)
case, nor to more general anisotropic problems. We hope however that
the analysis presented here, even though in a special case, can be
extended to handle more general anisotropic problems.

We would like to point out that our work was partially motivated by
some recent theoretical results for nearly singular problems
(see~\cite{Lee-Wu-Xu-Zikatanov2007}).  Indeed, anisotropic diffusion
equation gets more nearly singular when anisotropy gets
smaller. Techniques such as line smoother or semi-coarsening
correspond in some way to space splittings of the so called near
kernel components of the anisotropic diffusion problem.  We refer
to~\cite{Lee-Wu-Xu-Zikatanov2007} for description of such splittings.

The rest of this paper is organized as follows. In
Section~\ref{section:Preliminaries} we introduce the notation and
preliminaries. In Section~\ref{section:Theorem} we state the main
result. We then prove several stability and interpolation estimates
for coarse grid interpolant in Section~\ref{section:StabilityC} and
for the fine grid interpolant in Section~\ref{section:StabilityF}. In
section~\ref{section:Proof} we prove the main theorem, already stated
in section~\ref{section:Theorem}. Numerical experiments, which verify
the theory are given in section~\ref{section:Numerical}.

\section{Preliminaries and notation\label{section:Preliminaries}}
Consider the anisotropic diffusion equation on a square domain $\Omega\subset \mathbb{R}^2$:
 \begin{equation}\label{equation:model1}
 \left\{
 \begin{array}{rl}
 -u_{xx}-\epsilon u_{yy}=f,  &\mbox{in}\quad\Omega,\\
                       u=0,  &\mbox{on}\quad \partial\Omega,\\
 \end{array}
 \right.
 \end{equation}
 where $\epsilon>0$ is a constant. We are interested in the case
 when $\epsilon\rightarrow 0$. The weak formulation of (\ref{equation:model1})
 is: Find $u\in H_{0}^{1}(\Omega)$ such that
 \begin{equation}\label{equation:model2}
 a(u,v)=(f,v), \quad \forall v\in H_{0}^{1}(\Omega),
 \end{equation}
 where
 \begin{equation*}
 a(u,v)=\int_{\Omega}(\partial_{x}u\partial_{x}v+\epsilon\partial_{y}u\partial_{y}v)dxdy,
 \quad and\quad (f,v)=\int_{\Omega}(fv)dxdy.
 \end{equation*}

 We consider family of computational domains $\Omega$ obtained by
 rotations of a fixed domain $\Omega_0=(-1,1)^2$ around the
 origin. The angles of rotation are denoted by $\omega$ and we
 consider $\omega\in[0,\pi]$, since this covers all the possible cases
 of alignment (non-alignment) of the anisotropy and the FE grid.

 We assume that we have initial triangulation $\Tri{0}$ of the domain
 $\Omega_{0}$, obtained by dividing $\Omega_{0}$ into $N\times N$
 equal squares and then dividing every square into two triangles. Then
 we rotate $\Omega_0$ around the origin to obtain the computational
 domain $\Omega$ and its triangulation $\Tri{h}$. The finite element
 function space associated with $\Omega$ and $\Tri{h}$ will be the
 space of piece-wise continuous linear functions with respect to
 $\Tri{h}$ and we denote this space by $V_h$. One may see three such
 domains shown in Figure~\ref{fig:three-domains}. In such setting one
 case of grid aligned anisotropy corresponds to $\Omega=\Omega_0$, or
 equivalently, $\omega=0$.

%\begin{figure}[!h]
% \begin{center}
% \setlength{\unitlength}{0.8cm}
% \begin{picture}(20,5)
% \put(0,0){\line(1,0){4}}
% \put(0,1){\line(1,0){4}}
% \put(0,2){\line(1,0){4}}
% \put(0,3){\line(1,0){4}}
% \put(0,4){\line(1,0){4}}
%
% \put(0,0){\line(0,1){4}}
% \put(1,0){\line(0,1){4}}
% \put(2,0){\line(0,1){4}}
% \put(3,0){\line(0,1){4}}
% \put(4,0){\line(0,1){4}}
%
% \put(0,0){\line(1,1){4}}
% \put(1,0){\line(1,1){3}}
% \put(2,0){\line(1,1){2}}
% \put(3,0){\line(1,1){1}}
%
% \put(0,1){\line(1,1){3}}
% \put(0,2){\line(1,1){2}}
% \put(0,3){\line(1,1){1}}
%
% \put(5,2){\vector(1,0){1}}
%
% \put(9.5,-0.8){\line(1,1){2.83}}
% \put(9.5,-0.8){\line(-1,1){2.83}}
% \put(9.5,4.856){\line(1,-1){2.83}}
% \put(9.5,4.856){\line(-1,-1){2.83}}
%
% \put(10.2075,-0.0925){\line(-1,1){2.83}}
% \put(10.9150,0.6150){\line(-1,1){2.83}}
% \put(11.6225,1.3225){\line(-1,1){2.83}}
%
% \put(8.7925,-0.0925){\line(1,1){2.83}}
% \put(8.0805,0.6150){\line(1,1){2.83}}
% \put(7.3775,1.3225){\line(1,1){2.83}}
%
% \put(9.5,-0.8){\line(0,1){5.66}}
%
% \put(10.2075,-0.0925){\line(0,1){4.245}}
% \put(10.9150,0.6150){\line(0,1){2.83}}
% \put(11.6225,1.3225){\line(0,1){1.415}}
%
% \put(8.7925,-0.0925){\line(0,1){4.245}}
% \put(8.0805,0.6150){\line(0,1){2.830}}
% \put(7.3775,1.3225){\line(0,1){1.415}}
%
% \end{picture}
% \end{center}
% \vspace{0.5cm}
% \caption{The domain on the left corresponds to $\omega=0$,
%and the one on the right corresponds to $\omega=\frac{\pi}{4}$.
% % The triangular partition $\Tri{0}$ (i.e. $\omega=0$)
% }
%\label{fig:two domains}
% \end{figure}

\begin{figure}[!h]
\begin{center}
  \scalebox{0.33}[0.33]{\includegraphics{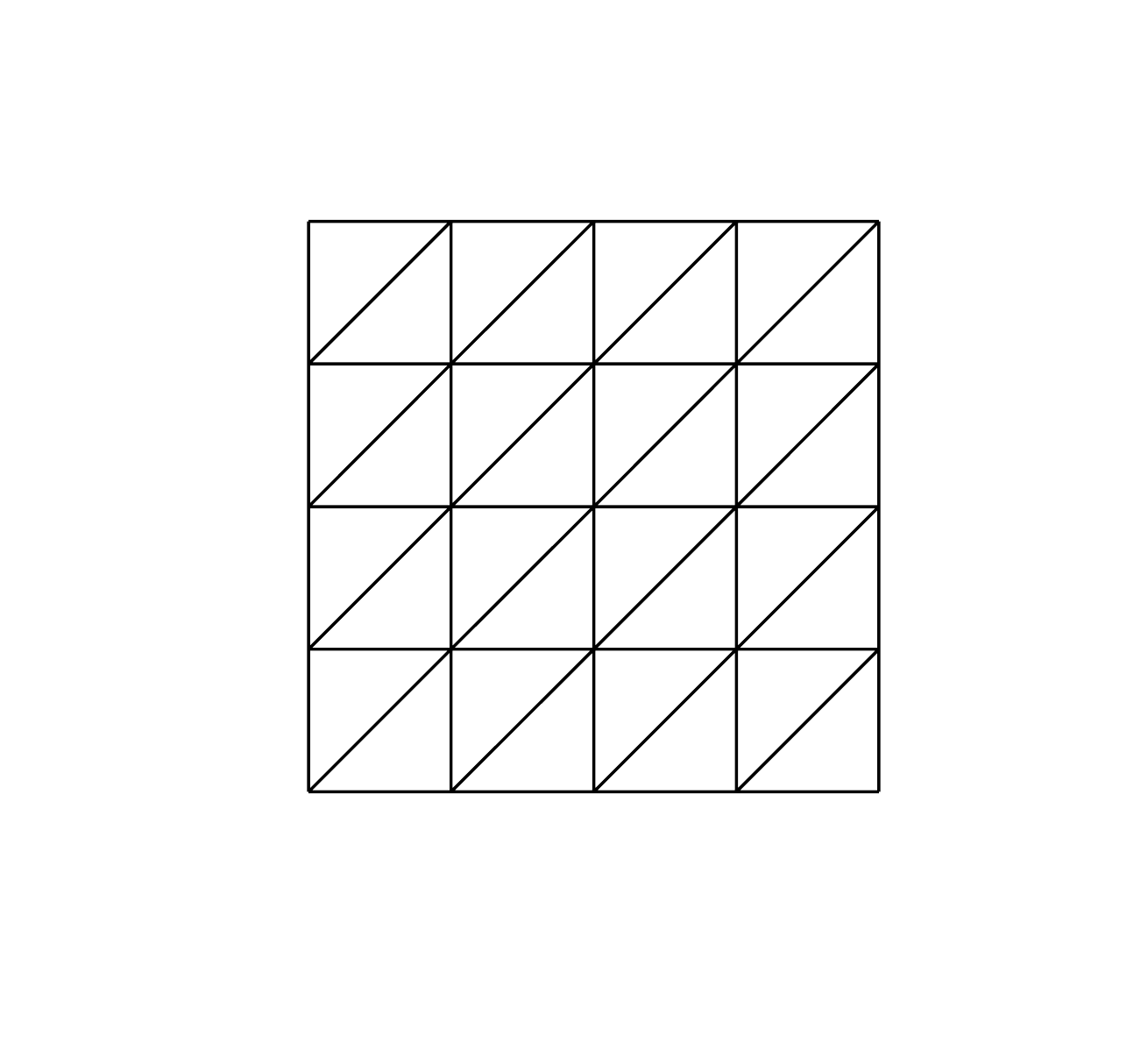}}
 \scalebox{0.33}[0.33]{\includegraphics{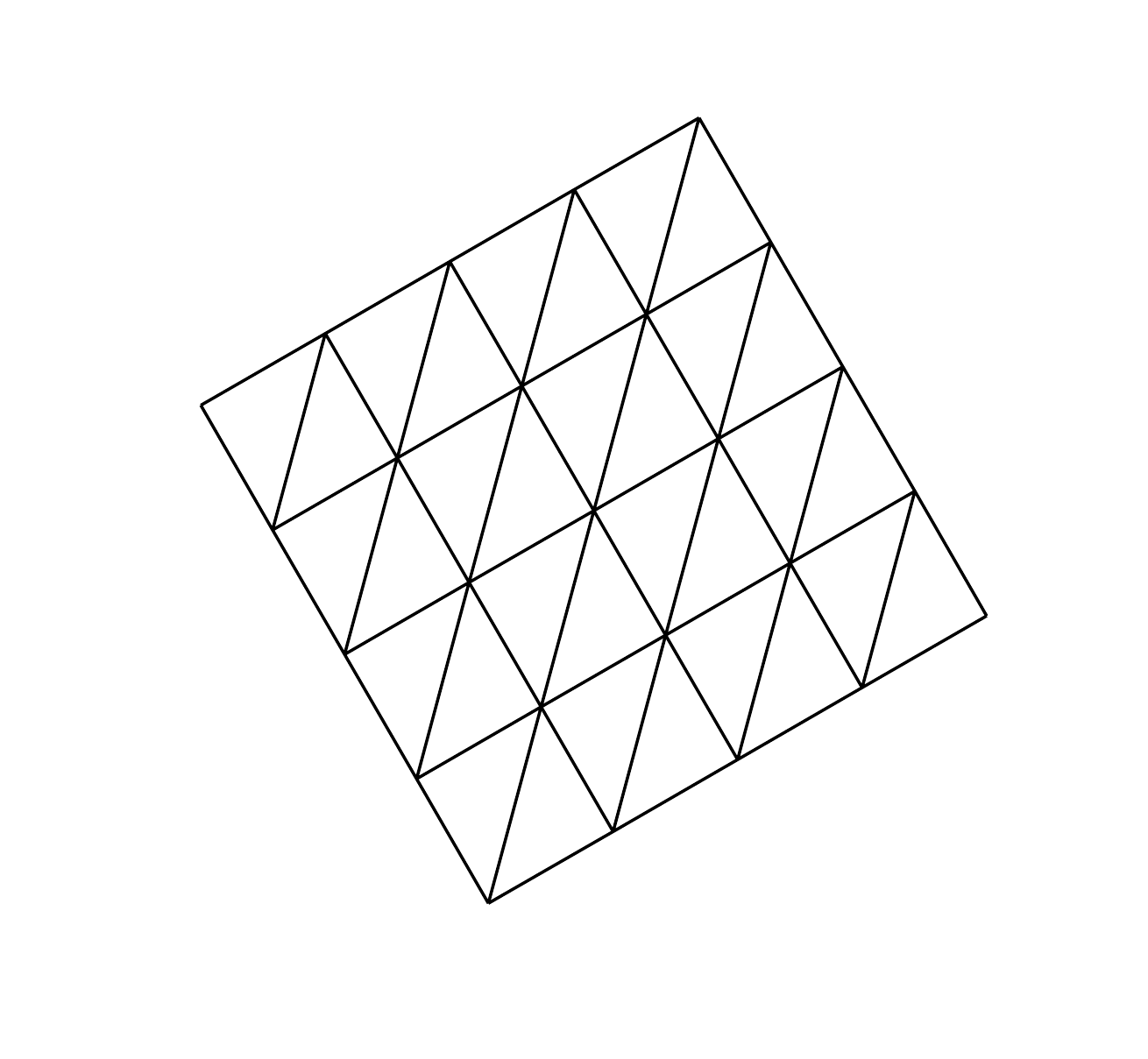}}
  \scalebox{0.33}[0.33]{\includegraphics{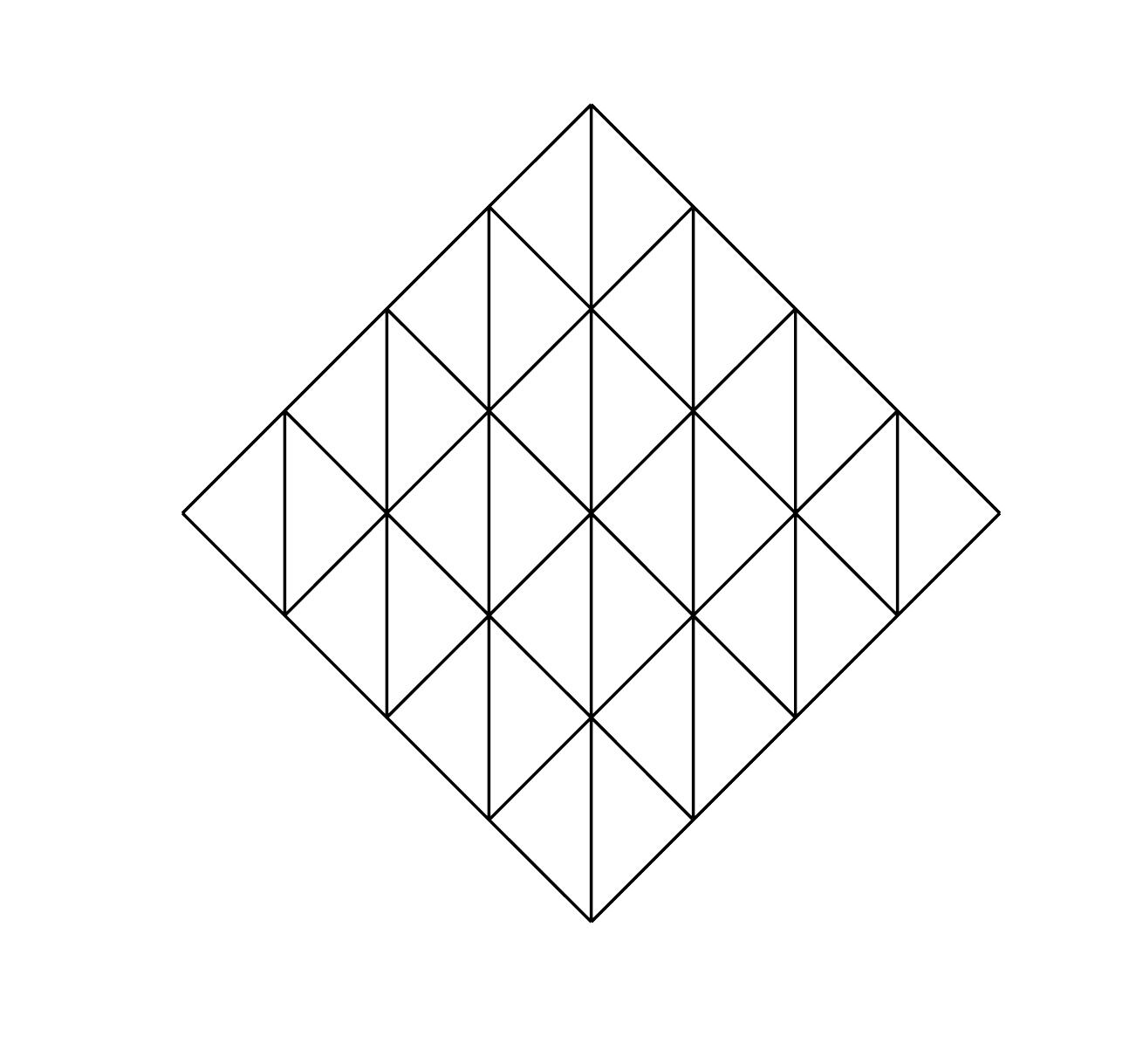}}
  \caption{The domains from left to right are corresponding to $\omega=0$,
  $\omega=\frac{\pi}{6}$ and $\omega=\frac{\pi}{4}$ respectively.
  \label{fig:three-domains}}
\end{center}
\end{figure}

 %\begin{figure}[!h]
% \begin{center}
% \setlength{\unitlength}{0.6cm}
% \begin{picture}(5,5)
% \put(0,0){\line(1,1){4}}
% \put(0,0){\line(-1,1){4}}
% \put(-4,4){\line(1,1){4}}
% \put(4,4){\line(-1,1){4}}
%
% \put(1,1){\line(-1,1){4}}
% \put(2,2){\line(-1,1){4}}
% \put(3,3){\line(-1,1){4}}
%
% \put(-1,1){\line(1,1){4}}
% \put(-2,2){\line(1,1){4}}
% \put(-3,3){\line(1,1){4}}
%
% \put(-3,3){\line(0,1){2}}
% \put(-2,2){\line(0,1){4}}
% \put(-1,1){\line(0,1){6}}
% \put(0,0){\line(0,1){8}}
% \put(1,1){\line(0,1){6}}
% \put(2,2){\line(0,1){4}}
% \put(3,3){\line(0,1){2}}
%
% \end{picture}
% \end{center}
% \caption{The triangular partition $\Tri{0}$ (i.e. $\omega=\frac{\pi}{4}$)}
% \end{figure}

 Given a coarse mesh $\Tri{H}$, assume the fine mesh $\Tri{h}$ is
 obtained from $\Tri{H}$ by splitting each of the triangles in the
 triangulation $\Tri{H}$ into four congruent triangles. One clearly
 then has $h=H/2$.  The spaces of continuous piece-wise linear
 functions corresponding to the partitions $\Tri{h}$ and $\Tri{H}$
are denoted by $V_{h}$ and $V_{H}$. As it is customary $I_{h}$, $I_{H}$ will denote the nodal interpolation operators mapping to $V_{h}$ and
$V_{H}$  respectively.

 For the analysis of the two-level method,
 we introduce partition of unity $\{\theta_{i}(y)\}_{i=1}^{L}$,
where $L=\left[\frac{y_{max}-y_{min}}h\right]$, with $y_{max}=\max\limits_{(x,y)\in \Omega}\{y\}$, $y_{min}=\min\limits_{(x,y)\in \Omega}\{y\}$.
We define $\theta_i$ as follows:
 \begin{equation}\label{eq:PU}
  \theta_{i}(y)=\left\{
  \begin{array}{ll}
  \frac{(y-y_{min})-(i-1)h}{h},  &(i-1)h\leq y-y_{min}\leq ih,\\
  \frac{(i+1)h-(y-y_{min})}{h},  &ih\leq y-y_{min}\leq (i+1)h,\\
  0,     &\mbox{other}.\\
  \end{array}
  \right.
 \end{equation}
 Note that each $\theta_i$ is piece-wise linear in the $y$ variable
 and is constant in $x$. Moreover, each $\theta_i$ is supported in the
 $i$-th strip $(i-1)h\leq y-y_{min}\leq (i+1)h$.

Denote the set of triangles in $\Tri{h}$ including nodes in the $i$-th strip by $\Tri{i}$, and Let
\begin{equation*}
\Omega_{i}=\bigcup_{\tau\in \Tri{i}}\tau,
\end{equation*}
\begin{equation*}
V_{i}=\{v\in V_h, supp\ v\subseteq \bar{\Omega}_{i}\},
\end{equation*}
then the two-level method with line smoother $V_{i} (1\leq i\leq L)$ and coarse grid $V_{H}$ can be written as
\begin{equation*}
V_{h}=\sum_{i=1}^{L}V_{i}+V_{H}.
\end{equation*}

Let $K$ be a triangle with vertices $\{(x_i,y_i)\}_{i=1}^3$, which we assume ordered counter-clockwise. For a given edge $E\in \partial K$, with
\[
E=((x_{i},y_{i}),(x_{j},y_{j})), \quad j=1+\operatorname{mod}(i,3),
\]
we denote
\begin{equation}\label{eq:deltaE}
\delta^K_{E}y=\frac{1}{2|K|}(y_{j}-y_{i}),
\end{equation}
we also denote with $(x_{E},y_{E})$ the coordinates of the vertex of $K$ which is opposite to $E$. In another word, if
$E=((x_{i},y_{i}),(x_{j},y_{j})) $ then $(x_{E},y_{E})=(x_k,y_k)$, where $k\neq i$ and $k\neq j$.  Let $v$ be a linear function on $K$. If we
set $v_{E}^{K}=v(x_{E},y_{E})$, then, it is easy to check that
\begin{equation}\label{eq:derivative identity}
\frac{\partial v}{\partial x}\bigg|_{K}=\sum_{E\in
\partial K}(\delta^K_{E}y)v_{E}^{K}.
\end{equation}

\section{Convergence of the two-level method\label{section:Theorem}}

We will first prove that even in the aligned case the most common point-wise smoothers will result in two-level method whose convergence
deteriorates when $\epsilon$ tends to zero in equation~\eqref{equation:model1}. The result is as follows.
\begin{theorem}[Lower bound and deterioration of the convergence rate]\label{lemma: counterexample}
In case of grid aligned anisotropy (i.e. $\omega=0,\frac{\pi}{2}, \frac{3\pi}{4}$) the energy norm of the error propagation operator
corresponding to the two-level iteration with coarse space $V_H$ and point-wise Gauss-Seidel smoother can be bounded below as follows:
\begin{equation}\label{equation: lower bound on E}
\|E_{TL}\|_{a}^{2}\geq 1-C(\epsilon+h^2),
\end{equation}
with constant $C$ independent of $\epsilon$ and $h$.
\end{theorem}
This result follows from the following two-level convergence identity (proof can be found in \cite[Lemma~2.3]{Zikatanov2008}):
\begin{lemma}\label{lemma: two-level by ludmil}
The following relation holds for the two-level error propagation operator $E_{TL}=(I-T)(I-P_{H})$:
\begin{equation}
\|E_{TL}\|_{a}^{2}=1-\frac{1}{K}\quad \mbox{where}\quad K=\sup_{v\in V}\frac{\|(I-\Pi_{*})v\|_{*}^{2}}{\|v\|_{a}^{2}},
\end{equation}
where $\|v\|_{*}^{2}=\inf\limits_{\sum_{i}v_{i}=v}\sum\limits_{k}\|v_{k}\|_{a}^{2}$ and $\Pi_{*}$ is an $(\cdot,\cdot)_{*}$-orthogonal
projection on $V_{H}$.
\end{lemma}

\begin{proof}[Proof of Theorem~\ref{lemma: counterexample}]
From the Lemma~\ref{lemma: two-level by ludmil} we can immediately see that to prove the estimate~\eqref{equation: lower bound on E} we need to
show that
\begin{equation}
K=\sup_{v\in V_h}\frac{\|(I-\Pi_{*})v\|_{*}^{2}}{\|v\|_{a}^{2}}\gtrsim \frac{1}{\epsilon+h^2},
\end{equation}
here the quantity $K$ is the same as in Lemma~\ref{lemma: two-level by ludmil}.
 From the proof of \cite[Theorem 4.5]{Zikatanov2008}, we also know
 that
\begin{equation}
K \gtrsim h^{-2} \sup_{v\in V_h}\frac{\|(I-Q_H)v\|^{2}}{\|v\|_{a}^{2}},
\end{equation}
where $Q_H$ is the $(\cdot,\cdot)$ orthogonal projection on $V_H$. In the case of angle of rotation $\omega=0$, the computational domain is
$\Omega=\Omega_0=(-1,1)^2$. We assume $h=1/n$ ($n$ is even) and consider the $2n\times 2n$ partition with vertices $(x_j,y_k)$, $x_j=jh$ and
$y_k=kh$, $j=-n,\cdots,n, k=-n,\cdots,n$, then the corresponding coarse gird is with vertices $(x_{2j},y_{2k})$, $j=-n/2,\cdots,n/2$,
$k=-n/2,\cdots,n/2$.

For any given $v\in V_h$, since
$$\|v\|^2\simeq h^2
\sum\limits_{j=-n}^n\sum\limits_{k=-n}^nv^2(x_j,y_k),
$$
and
$$\|I_Hv\|^2\simeq H^2
\sum\limits_{j=-n/2}^{n/2}\sum\limits_{k=-n/2}^{n/2}v^2(x_{2j},y_{2k}),
$$
then the interpolation $I_H$ is stable in the $L_2$ norm, i.e. $\|I_Hv\|\lesssim\|v\|$, for all $v\in V_h$, provided that $\frac{H}{h}\lesssim
1$. Now consider a function $v_0\in V_h$, supported in the closure of $(-1,1)\times(0,2h)$ and defined as
$$
v_0(x_j,y_{1})=v_0(x_j,h)=1-|j|h, j=-n,\cdots,n,
$$
and $v_0$ is 0 at any other vertex. Note that
\[
I_H v_0 = 0.
\]

From the stability of $I_H$ in the $L_2$ norm, which we have just shown we get:
\begin{eqnarray*}
\|I_H(I-Q_H)v_0\|^2 & \lesssim & \|(I-Q_H)v_0\|^2,\\
\|(I-I_H)(I-Q_H)v_0\|^2 &\lesssim&  \|(I-Q_H)v_0\|^2.
\end{eqnarray*}
Using these estimates and the fact that $I_HQ_H=Q_H$ then gives
\begin{eqnarray*}
\|(I-Q_H)v_0\|^{2} &\gtrsim &
\|I_H(I-Q_H)v_0\|^2+\|(I-I_H)(I-Q_H)v_0\|^2\\
&= &\|(I_H-Q_H)v_0\|^2+\|(I-I_H)v_0\|^2\\
&= &\|Q_Hv_0\|^2+\|v_0\|^2\\
&\geq &\|v_0\|^2.
\end{eqnarray*}
So
\begin{equation*}
K \gtrsim h^{-2} \sup_{v\in V_h}\frac{\|(I-Q_H)v\|^{2}}{\|v\|_{a}^{2}}\gtrsim h^{-2} \frac{\|(I-Q_H)v_0\|^{2}}{\|v_0\|_{a}^{2}}\gtrsim
h^{-2}\frac{\|v_0\|^{2}}{\|v_0\|_{a}^{2}}.
\end{equation*}
Since
\begin{equation*}
\|v_0\|^{2}\simeq h^2\sum_{j=-n}^n v_0^2(x_j,h)\simeq h^2\sum_{j=0}^n (1-jh)^2=h^4\sum_{j=0}^n j^2\simeq h^4n^3\simeq h,
\end{equation*}
and
\begin{equation*}
\|v_0\|_a^{2}=\|\partial_x v_0\|^2+\epsilon \|\partial_y v_0\|^2\simeq h+\epsilon/h,
\end{equation*}
then
\begin{equation*}
K \gtrsim h^{-2}\frac{\|v_0\|^{2}}{\|v_0\|_{a}^{2}} \gtrsim h^{-2}\cdot h\cdot \frac{h}{\epsilon+h^2}=\frac{1}{\epsilon+h^2}.
\end{equation*}
\end{proof}

The above results show that in case of a grid that is \emph{aligned} with the anisotropy direction the convergence of a standard two-level
method (point-wise smoother and standard coarsening) will deteriorate. One easily sees that for $\epsilon \le h^2$ we get a poor convergence
rate (no better than $1-\mathcal{O}(h^2)$).

However, the next result shows that when the grid is not aligned with the anisotropy direction (e.g., angle of rotation $\omega=\pi/4)$ the
lower bound given in Lemma~\ref{lemma:
  counterexample} does not apply and the standard two-level method is
uniformly convergent in this case.

\begin{lemma}
Assume that $\Omega$ is obtained from $\Omega_0$ by a rotation with angle of rotation $\omega=\frac{\pi}{4}$. Then the error propagation
operator corresponding to the two-level iteration with coarse space $V_H$ and point-wise Gauss-Seidel smoother is a uniform contraction in the
energy norm. In fact, we have the estimate
\begin{equation}
K=\sup_{v\in V_h}\frac{\|(I-\Pi_{*})v\|_{*}^{2}}{\|v\|_{a}^{2}}\leq C,
\end{equation}
and so
\begin{equation}
\|E_{TL}\|_{a}^{2}\leq 1-\frac{1}{C},
\end{equation}
with constant $C$ independent of $\epsilon$ and $h$.
\end{lemma}
\begin{proof}
The proof follows the same lines as the proof of Theorem \ref{theorem: the main theorem}, which will be given in the
Section~\ref{section:Proof}.
\end{proof}

Let us remark here that for decreasing values of the angle of rotation (i.e. decreasing $\omega$ from $\pi/4$ to $0$) the convergence rate
$\|E_{TL}\|_a$ of a two-level method with point-wise smoother deteriorates as the angle of rotation becomes smaller.

From the above considerations, it is clear that even in the case of aligned anisotropy one needs to use a special smoother or coarsening
strategy in order to achieve uniform convergence. Our analysis shows that the line smoother or more generally a block smoother with blocks
consisting of degrees of freedom along  the anisotropy results in a uniformly convergent method. The Theorem below provides a uniform estimate
on the convergence rate of the error propagation operator and is the main result in this paper.
\begin{theorem}\label{theorem: the main theorem}
For any angle of rotation $\omega\in [0,\pi]$, the two-level iteration with coarse space $V_H$ and line (block) Gauss-Seidel smoother is a
uniformly convergent method. In fact, we have
\begin{equation}\label{eq:convergence rate}
 \|E_{TL}\|_{a}^{2}\leq 1-\frac{1}{C},
 \end{equation}
with constant $C$ independent of $\epsilon$ and $h$.
\end{theorem}
The proof of this theorem is postponed to Section~\ref{section:Proof}. The result follows from the stability and interpolation estimates that
are given in Section~\ref{section:StabilityC} and Section~\ref{section:StabilityF} and Lemma~\ref{lemma: two-level by ludmil}.

\section{Stability of the coarse grid interpolant\label{section:StabilityC}}
In this section we prove the stability of the coarse grid interpolant.

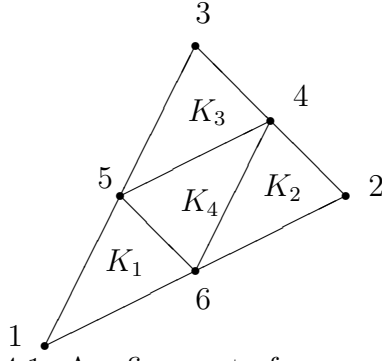
\begin{figure}[!h]
 \begin{center}
 \setlength{\unitlength}{1.0cm}
 \begin{picture}(5,5)
 \put(0,0){\line(1,2){2}}
 \put(0,0){\line(2,1){4}}
 \put(4,2){\line(-1,1){2}}
 \put(1,2){\line(2,1){2}}
 \put(1,2){\line(1,-1){1}}
 \put(2,1){\line(1,2){1}}

 \put(0,0){\circle*{0.1}}\put(-0.5,0){1}
 \put(4,2){\circle*{0.1}}\put(4.3,2){2}
 \put(2,4){\circle*{0.1}}\put(2,4.3){3}
 \put(3,3){\circle*{0.1}}\put(3.3,3.2){4}
 \put(1,2){\circle*{0.1}}\put(0.7,2.1){5}
 \put(2,1){\circle*{0.1}}\put(2,0.5){6}

 \put(0.8,1){$K_{1}$}
 \put(2.9,2){$K_{2}$}
 \put(1.9,3){$K_{3}$}
 \put(1.8,1.8){$K_{4}$}

 \end{picture}
 \end{center}
 \caption{A refinement of a coarse element $K$
 \label{fig:refinement}}
 \end{figure}

 In what follows, given a triangle $K\in \Tri{H}$ ($K=\bigcup\limits_{l=1}^4K_l$ with
$K_l\in\Tri{h}$) as shown in the
 Figure \ref{fig:refinement}, we shall frequently use the following equalities
 \begin{equation}\label{eq:triangle equalities}
 \begin{array}{l}
 y_{2}-y_{6}=y_{6}-y_{1}=y_{4}-y_{5},\\
 y_{3}-y_{4}=y_{4}-y_{2}=y_{5}-y_{6},\\
 y_{1}-y_{5}=y_{5}-y_{1}=y_{6}-y_{4},\\
 |K|=4|K_{l}|,\quad l=1\ldots 4.
 \end{array}
 \end{equation}

 \begin{proposition}\label{proposition: relations}
  For any $v\in V_h$, we have the following relation:
\begin{equation}\label{eq:identity}
\frac{\partial (I_{H}v)}{\partial x}\bigg|_{K}=\frac{1}{2}\left(\sum_{l=1}^{3}\frac{\partial v}{\partial x}\bigg|_{K_{l}}-\frac{\partial
v}{\partial x}\bigg|_{K_{4}}\right).
\end{equation}
 \end{proposition}
 \begin{proof}
 From \eqref{eq:derivative identity} with $I_{H}v$ instead of $v$, we have
 \begin{equation*} %\label{eq: derivative of interpolation}
 \frac{\partial (I_{H}v)}{\partial   x}\bigg|_{K}=\sum_{E\in \partial K} (\delta_E^K y)
 v_{E}^{K},
 \end{equation*}
 and from \eqref{eq:derivative identity} with $K_l$ instead of $K$ for $l=1,\ldots,4$, we have
  \begin{equation*} %\label{eq: derivative of Vh}
 \frac{\partial v}{\partial x}\bigg|_{K_{l}}
=\sum_{E\in \partial K_l} (\delta_E^{K_l} y)v_{E}^{K_{l}}.
\end{equation*}
 Combine the above two equations and \eqref{eq:triangle equalities}, it is immediate to verify
 the result.
 \end{proof}

 We are now ready to prove our first stability estimate. Since we have
 anisotropic diffusion problem in hand, we need to estimate separately
 $\left\|\frac{\partial (I_{H}v)}{\partial x}\right\|_{0}$ and
 $\left\|\frac{\partial (I_{H}v)}{\partial y}\right\|_{0}$, which is
 done in the next Lemma.
\begin{lemma}\label{lemma: two-level}
 For any $v\in V_{h}$,
 we have $\left\|\frac{\partial (I_{H}v)}{\partial x}\right\|_{0}^{2}\leq 4 \|\frac{\partial v}{\partial
 x}\|_{0}^{2}$, and $\left \|\frac{\partial (I_{H}v)}{\partial y}\right\|_{0}^{2}\leq 4 \|\frac{\partial v}{\partial
 y}\|_{0}^{2}.$
\end{lemma}
 \begin{proof}
We only need to prove this estimate locally for any $K\in \Tri{H}$. So we fix $K\in \Tri{H}$ and we would like to show that
 $\|\frac{\partial (I_{H}v)}{\partial x}\|_{0,K}^{2}\leq 4 \|\frac{\partial v}{\partial
 x}\|_{0,K}^{2}$.

 From \eqref{eq:identity}, for the $L_2$ norm $ \|\frac{\partial
(I_{H}v)}{\partial x}\|_{0,K}^{2} $ we have

\begin{eqnarray*}
\left\|\frac{\partial (I_{H}v)}{\partial x}\right\|_{0,K}^{2}
 &=&\frac{|K|}{4}\left(\sum\limits_{l=1}^{3}\frac{\partial v}{\partial x}\bigg|_{K_{l}}-\frac{\partial v}{\partial x}\bigg|_{K_{4}}\right)^{2}\\
 &\leq &|K|\sum\limits_{l=1}^{4}\left(\frac{\partial v}{\partial x}\bigg|_{K_{l}}\right)^{2}
=4\sum\limits_{l=1}^{4}|K_{l}|\left(\frac{\partial v}{\partial x}\bigg|_{K_{l}}\right)^{2}\\
&=&4\sum\limits_{l=1}^{4}\left\|\frac{\partial v}{\partial x}\right\|_{0,K_{l}}^{2}=4\left\|\frac{\partial v}{\partial x}\right\|_{0,K}^{2}.
 \end{eqnarray*}

Summing over all the elements then gives
 $\|\frac{\partial (I_{H}v)}{\partial x}\|_{0}^{2}\leq 4\|\frac{\partial v}{\partial
 x}\|_{0}^{2}.$
In a similar fashion we can prove that $\|\frac{\partial (I_{H}v)}{\partial y}\|_{0}^{2}\leq 4\|\frac{\partial v}{\partial
 y}\|_{0}^{2},$
and the proof of the lemma is complete.
\end{proof}

As a consequence, we have the following approximation result for the coarse grid interpolant.

\begin{lemma} \label{lemma: two level convergence}
 For any $v\in V_{h}$,
 we have
 $$
 \left\|\frac{\partial (v-I_{H}v)}{\partial x}\right\|_{0}^{2}\lesssim  \left\|\frac{\partial v}{\partial
 x}\right\|_{0}^{2},
 \quad
 \left\|\frac{\partial (v-I_{H}v)}{\partial y}\right\|_{0}^{2}\lesssim  \left\|\frac{\partial v}{\partial
 y}\right\|_{0}^{2},
 $$
 and
 $$
 \|v-I_{H}v\|_{0}^{2} \lesssim h^{2} |v|_{1}^{2}.
 $$
\end{lemma}
\begin{proof}
The first two estimates follow from the inequalities given in Lemma \ref{lemma: two-level}. The third estimate can be found in \cite[Lemma
4.4]{Bramble-Xu1991}.
\end{proof}
\begin{remark} In fact, in the proof of Lemma \ref{lemma: two-level},
 there is no any requirement for the partition. The result is
true for partition $\Tri{h}$ obtaining from regular refinement of any given partition $\Tri{H}$. So is the Lemma \ref{lemma: two level
convergence}.
\end{remark}

\section{Stability estimates on the fine grid\label{section:StabilityF}}

In this section we give estimates on the stability of the partition of unity introduced in Section ~\ref{section:Preliminaries},
equation~\eqref{eq:PU}. In what follows, to avoid proliferation of indicies, we will omit the subscript $i$ and we will write $\theta$ instead
of $\theta_i$.

For any given $K\in\Tri{H}$ ($K=\bigcup\limits_{l=1}^4K_l$ with $K_l\in\Tri{h}$) shown in Figure~\ref{fig:refinement}, we label with 1, 2, and 3
the vertices of $K$, and 4, 5, and 6 the midpoints of $K$ (these are also vertices of $K_4$). The corresponding coordinates are denoted by
$(x_{j},y_{j}), 1\leq j\leq 6$. Further, for a continuous function $v$, when there is no confusion, we write $v_{j}:=v(x_{j},y_{j})$.

In what follows we also denote
 \begin{equation}\label{eq: minmum of edge}
 E_{min}=arg\min\limits_{E\in\partial
 K_{4}}\{|\delta_{E}^{K_{4}}y|\}, \quad
\mbox{where $\delta_{E}^{K_4}y$ is defined in~\eqref{eq:deltaE}}.
 \end{equation}
 Then $E_{min}$ is related to the anisotropic direction. In fact, to
 indicate the dependence on the particular element, one may write
 $E_{min}^{K_4}$ instead of $E_{min}$, but for simplicity we have
 chosen to omit the superscript $K_4$. Furthermore, we may denote
 $E_{min}'=arg\min\limits_{E'\in\partial
   K'_{4}}\{|\delta_{E'}^{K'_{4}}y|\}$.

Let us now consider a function $w \in V_{h}$ vanishing at the coarse grid vertices, that is, $w$ satisfies $I_H w=0$. We have the following 4
cases on a fixed $K\in\Tri{H}$:
\renewcommand{\labelenumi}{Case \arabic{enumi}.}
\begin{enumerate}
\addtocounter{enumi}{-1}
\item $\theta$ is zero in $K_4$;

\item $\theta$ is nonzero in $K_4$ and convex in $K_4$
  (i.e. $\theta\neq 0$ a.e. in $K_4$);

\item $\theta$ is nonzero at only one of the vertices of $K_4$ and concave
  in $K_4$;

\item $\theta$ is nonzero at exactly two of the vertices of $K_4$ and concave
  in $K_4$.
\end{enumerate}

The rest of this section contains technical results and their proofs, which can be classified according to the cases above. To prove the
stability estimates on the fine grid, we need to bound $\|\frac{\partial (I_{h}(\theta w))}{\partial x}\|$.

\begin{itemize}
\item For the Case 0, there is nothing to prove, since in this case $I_{h}(\theta w)=0$.

\item For the Case 1 the corresponding estimate is given in Lemma
\ref{lemma: 3 point stability}. In Case 1 we also need to assume quasi-uniformity of the mesh.  We also note that Proposition~\ref{proposition:
3 point property} (for Case 1) contains an estimate which is later used in Case 2 and Case 3.

\item The stability estimates in Case 2 and Case 3 are given in Lemmas
\ref{lemma: 1 point stability} and \ref{lemma: 2 point
  stability}, respectively, under the assumption that $\Tri{h}$ is a
uniform partition.
\end{itemize}
\begin{remark} In summary, for uniform mesh we have proved the stability
estimate in all cases. In addition, we have proved some of the results in more general case of unstructured, but quasi-uniform mesh (Case 1).
\end{remark}

 \begin{lemma}\label{lemma: 3 point stability}
 Assume that $\theta\neq 0$ in $K_4$ and convex in $K_4$ (Case 1). Then
  $$\left\|\frac{\partial (I_{h}(\theta w))}{\partial x}\right\|_{0,K}^{2}\lesssim \left\|\frac{\partial (\theta w)}{\partial
 x}\right\|_{0,K}^{2}.$$
 \end{lemma}

 \begin{proof}
   Since $I_Hw=0$ we obviously have that $I_H(\theta w)=0$ as
   well. From \eqref{eq:identity} in Proposition \ref{proposition: relations}
with $v=I_h(\theta w)$ we obtain that
\begin{equation}\label{eq:51}
 \sum_{l=1}^{3}\frac{\partial (I_h(\theta w))}{\partial x}\bigg|_{K_{l}}-\frac{\partial (I_h(\theta w))}{\partial x}\bigg|_{K_{4}}
 =2\frac{\partial(I_H (I_h(\theta w)))}{\partial x}\bigg|_{K}=2\frac{\partial(I_H(\theta w))}{\partial x}\bigg|_{K}=0.
 \end{equation}
In addition, from \eqref{eq:identity}, with $v=w$ we have
 \begin{equation}\label{eq:52}
 \sum_{l=1}^{3}\frac{\partial w}{\partial x}\bigg|_{K_{l}}-\frac{\partial w}{\partial x}\bigg|_{K_{4}}
 =2\frac{\partial(I_H w)}{\partial x}\bigg|_{K}=0.
 \end{equation}

Therefore, from~\eqref{eq:51} for the $L^{2}$ norm $\|\frac{\partial (I_{h}(\theta w))}{\partial x}\|_{0,K}$ we have:
 \begin{eqnarray}
 \left\|\frac{\partial (I_{h}(\theta w))}{\partial x}\right\|_{0,K}^{2}
&=&\sum\limits_{l=1}^{4}|K_{l}|\left(\frac{\partial(I_h (\theta w))}{\partial x}\bigg|_{K_l}\right)^{2}
 \thickapprox\sum\limits_{l=1}^{3}|K_{l}|\left(\frac{\partial(I_h (\theta w))}{\partial x}\bigg|_{K_l}\right)^{2}\nonumber\\
 &\thickapprox&\frac{1}{|K|}\{[(y_{4}-y_{6})\theta_{6}w_{6}+(y_{5}-y_{4})\theta_{5}w_{5}]^{2}\label{norm1}\\
                &&\,\,\quad+[(y_{5}-y_{4})\theta_{4}w_{4}+(y_{6}-y_{5})\theta_{6}w_{6}]^{2}\nonumber\\
                &&\,\,\quad+[(y_{6}-y_{5})\theta_{5}w_{5}+(y_{4}-y_{6})\theta_{4}w_{4}]^{2}\}.\nonumber
 \end{eqnarray}

 On the other hand, since $\theta$ is a convex function in $K_4$, and $\theta$ is
supported in a
   $2h$ width strip, $\theta$ should be convex in at least three of
   $K_l(l=1:4)$.

 For any $K_l$ in which $\theta$ is convex, we have
 $$
 \|\theta\|_{0,K_{l}}^{2}\gtrsim
 |K_l|\sum_{j=1}^{3}\theta^{2}(x_{j}^{K_{l}},y_{j}^{K_{l}}),
 $$
 where $(x_{j}^{K_{l}},y_{j}^{K_{l}})$ denote the coordinates of the
 $j$-th vertex of element $K_{l}$ for $j=1:3$.  Since the mesh is
 quasi-uniform, we have that
 $\max\limits_{j=1:3}\{y_j^{K_{l}}\}-\min\limits_{j=1:3}\{y_j^{K_{l}}\}\gtrsim
 h$, and there exists at least one vertex $j_0$ such that
 $\theta(x_{j_{0}},y_{j_0})=\theta(y_{j_0})\gtrsim 1$. Hence, if
 $\theta$ is a convex function in $K_{l}$, then
 $\|\theta\|_{0,K_{l}}^{2}\gtrsim|K_l|$, and we conclude that there
 are at least three elements $K_{l}$, where
 $\|\theta\|_{0,K_{l}}^{2}\gtrsim|K_l|$ holds.

 From this argument and ~\eqref{eq:52} we get
 \begin{eqnarray}
 \left\|\frac{\partial (\theta w)}{\partial x}\right\|_{0,K}^{2}
&=&\left\|\theta\frac{\partial w}{\partial x}\right\|_{0,K}^{2}
  = \sum\limits_{l=1}^{4}\|\theta\|_{0,K_{l}}^{2} \left(\frac{\partial w}{\partial x}\bigg|_{K_l}\right)^{2}
     \thickapprox\sum\limits_{l=1}^{3}|K_l|\left(\frac{\partial w}{\partial
     x}\bigg|_{K_l}\right)^{2}\nonumber\\
 &\thickapprox&\frac{1}{|K|}\{[(y_{4}-y_{6})w_{6}+(y_{5}-y_{4})w_{5}]^{2}\label{norm2}\\
                    &&\,\,\quad
                    +[(y_{5}-y_{4})w_{4}+(y_{6}-y_{5})w_{6}]^{2}\nonumber\\
                    &&\,\,\quad
                    +[(y_{6}-y_{5})w_{5}+(y_{4}-y_{6})w_{4}]^{2}\}.\nonumber
 \end{eqnarray}

 Introducing now
 \begin{equation*}
 M=\left(
 \begin{matrix}
 0 &y_{5}-y_{4} &y_{4}-y_{6}\\
 y_{5}-y_{4} &0 &y_{6}-y_{5}\\
 y_{4}-y_{6} &y_{6}-y_{5} &0\\
 \end{matrix}
 \right), \quad
\Theta=\left(
 \begin{matrix}
 \theta_{4} &&\\
 &\theta_{5} &\\
 &&\theta_{6}\\
 \end{matrix}
 \right),
 \end{equation*}
we rewrite~\eqref{norm1} as
%  \begin{equation*}
%  \|\frac{\partial (I_{h}(\theta w))}{\partial x}\|_{0,K}^{2}\approx
%  \frac{1}{|K|}
%  \left(
%  \begin{array}{c}
%  w_{4}\\
%  w_{5}\\
%  w_{6}\\
%  \end{array}
%  \right)'\Theta M^{2} \Theta
%  \left(
%  \begin{array}{c}
%  w_{4}\\
%  w_{5}\\
%  w_{6}\\
%  \end{array}
%  \right),
%  \end{equation*}
% while \eqref{norm2} can be rewritten as
%  \begin{equation*}
%  \|\frac{\partial (\theta w)}{\partial x}\|_{0,K}^{2}\thickapprox
%  \frac{1}{|K|}
%  \left(
%  \begin{array}{c}
%  w_{4}\\
%  w_{5}\\
%  w_{6}\\
%  \end{array}
%  \right)' M^{2}
%  \left(
%  \begin{array}{c}
%  w_{4}\\
%  w_{5}\\
%  w_{6}\\
%  \end{array}
%  \right).
%  \end{equation*}
 \begin{equation*}
 \left\|\frac{\partial (I_{h}(\theta w))}{\partial x}\right\|_{0,K}^{2}\approx
 \frac{1}{|K|}\|M\Theta\bm{z}\|^2_{\ell_2}, \quad
\bm{z}= (w_{4},  w_{5}, w_{6})^t,
\end{equation*}
while \eqref{norm2} can be rewritten as
 \begin{equation*}
 \left\|\frac{\partial (\theta w)}{\partial x}\right\|_{0,K}^{2}\approx
 \frac{1}{|K|}\|M\bm{z}\|^2_{\ell_2}, \quad
\bm{z}= (w_{4},  w_{5}, w_{6})^t.
\end{equation*}
Here, $\|\cdot\|_{\ell_2}$ is the usual Euclidean norm on $\mathbb{R}^3$.

 To prove the estimate $\|\frac{\partial (I_{h}(\theta w))}{\partial x}\|_{0,K}^{2}
 \lesssim\|\frac{\partial (\theta w)}{\partial x}\|_{0,K}^{2}$, we
 only need to show that
\begin{equation}\label{eq:haha}
\frac{1}{|K|}\|M\Theta\bm{z}\|^2_{\ell_2}\lesssim
 \frac{1}{|K|}\|M\bm{z}\|^2_{\ell_2}\quad\mbox{for all}\quad \bm{z}\in
 \mathbb{R}^3.
\end{equation}

Such an inequality is easy to get in the case of $\det(M)=0$, so we may assume $M$ is invertible (i.e.
$\det(M)=2(y_{6}-y_{5})(y_{4}-y_{6})(y_{5}-y_{4})\neq 0$). We then need a bound on the eigenvalues of $M^{-1}\Theta M^{2}\Theta
M^{-1}=(M^{-1}\Theta M) (M^{-1}\Theta M)^T$.  In fact, we only need to bound the entries of $M^{-1}\Theta M$ because all the norms of this
$3\times 3$ matrix are equivalent. Thus, if the entries of $M^{-1}\Theta M$ are bounded in absolute value, then the eigenvalues of
$(M^{-1}\Theta M) (M^{-1}\Theta M)^T$ are bounded and consequently \eqref{eq:haha} holds.

Directly computing the inverse of $M$ gives
 \begin{equation*}
 M^{-1}=
 \frac{1}{\det(M)}
 \left(
 \begin{array}{ccc}
 -(y_{6}-y_{5})^{2} &(y_{6}-y_{5})(y_{4}-y_{6}) &(y_{6}-y_{5})(y_{5}-y_{4})\\
  (y_{4}-y_{6}) (y_{6}-y_{5}) &-(y_{4}-y_{6})^{2} &(y_{4}-y_{6})(y_{5}-y_{4})\\
 (y_{5}-y_{4})(y_{6}-y_{5}) &(y_{5}-y_{4})(y_{4}-y_{6}) &-(y_{5}-y_{4})^{2}\\
 \end{array}
 \right).
 \end{equation*}
We then calculate   $M^{-1}\Theta M$ to obtain that
 \begin{equation*}
  M^{-1}\Theta M=\frac{1}{2}
 \left(
 \begin{array}{ccc}
 \theta_{5}+\theta_{6}
 &\frac{y_{6}-y_{5}}{y_{4}-y_{6}}(-\theta_{4}+\theta_{6})
 &\frac{y_{6}-y_{5}}{y_{5}-y_{4}}(-\theta_{4}+\theta_{5})\\
 \frac{y_{4}-y_{6}}{y_{6}-y_{5}}(-\theta_{5}+\theta_{6})
 &\theta_{4}+\theta_{6}
 &\frac{y_{4}-y_{6}}{y_{5}-y_{4}}(\theta_{4}-\theta_{5})\\
 \frac{y_{5}-y_{4}}{y_{6}-y_{5}}(\theta_{5}-\theta_{6})
 &\frac{y_{5}-y_{4}}{y_{4}-y_{6}}(\theta_{4}-\theta_{6})
 &\theta_{4}+\theta_{5}\\
 \end{array}
 \right).
 \end{equation*}
 Since $\theta$ is convex in $K_4$, by the definition of $\theta$, it
 is easy to see that
 $$
 |\theta_{6}-\theta_{5}|\lesssim h^{-1} |y_{6}-y_{5}|,
 $$
 $$
 |\theta_{5}-\theta_{4}|\lesssim h^{-1} |y_{5}-y_{4}|,
 $$
 $$
 |\theta_{4}-\theta_{6}|\lesssim h^{-1} |y_{4}-y_{6}|.
 $$
 Since $|y_{i}-y_{j}|\lesssim h$, we have $|(M^{-1}\Theta
 M)_{ij}|\lesssim 1$ and the proof of the Lemma is complete.
\end{proof}

 Next result is an auxiliary estimate used later in the proof of
 Lemma \ref{lemma: 1 point stability} and \ref{lemma: 2 point
 stability}. In the statement of the lemma we used the notation given
at the end of Section~\ref{section:Preliminaries}.

 \begin{proposition}\label{proposition: 3 point property}
  Assume that $\theta\neq 0$ and convex in $K_4$. Then the following
  inequality holds
 \begin{equation*}
\left\|\frac{\partial (\theta w)}{\partial x}\right\|_{0,K}^{2}\gtrsim |K| (\max\limits_{E\in\partial
  K_{4}}\{|\delta_{E}^{K_{4}}y|\}^{2}\cdot(w_{E_{min}}^{K_{4}})^{2}
  +(\delta_{E_{min}}^{K_{4}}y)^{2}\cdot
  \max\limits_{E\in\partial K_{4}}\{w_{E}^{K_{4}}\}^{2}).
 \end{equation*}
 \end{proposition}
 \begin{proof}
 Let $E_{min}$ be defined as \eqref{eq: minmum of edge}.
 Without loss of generality, assume
 $E_{min}=\{(x_4,y_4),(x_5,y_5)\}$, and then $w_6=w_{E_{min}}^{K_{4}}$.  This means
 $$|y_{5}-y_{4}|=\min\{|y_{5}-y_{4}|,|y_{4}-y_{6}|,|y_{6}-y_{5}|\}.$$
 Hence
\begin{equation*}
   \left|\frac{y_{5}-y_{4}}{y_{6}-y_{5}}\right|\leq 1,\quad
   \left|\frac{y_{5}-y_{4}}{y_{4}-y_{6}}\right|\leq 1,
\end{equation*}
and by triangle inequalities we have
\begin{equation*}
 \left|\frac{y_{4}-y_{6}}{y_{6}-y_{5}}\right|\leq 2,\quad
 \left|\frac{y_{6}-y_{5}}{y_{4}-y_{6}}\right|\leq 2.
\end{equation*}

 According to the expression (\ref{norm2}) and the above inequalities, we have
 \begin{eqnarray*}
 \left\|\frac{\partial (\theta w)}{\partial x}\right\|_{0,K}^{2}
 &\thickapprox&\frac{1}{|K|}\{[(y_{4}-y_{6})w_{6}+(y_{5}-y_{4})w_{5}]^{2}
                    +[(y_{5}-y_{4})w_{4}+(y_{6}-y_{5})w_{6}]^{2}\\
                    &&\,\,\quad
                    +[(y_{6}-y_{5})w_{5}+(y_{4}-y_{6})w_{4}]^{2}\}\\
 &\gtrsim
 &\frac{1}{|K|}
  \{[(y_{4}-y_{6})w_{6}+(y_{5}-y_{4})w_{5}]
     +[(y_{5}-y_{4})w_{4}+(y_{6}-y_{5})w_{6}]\frac{y_{4}-y_{6}}{y_{6}-y_{5}}\\
&&\,\,\quad-[(y_{6}-y_{5})w_{5}+(y_{4}-y_{6})w_{4}]\frac{y_{5}-y_{4}}{y_{6}-y_{5}}\}^{2}\\
 &=
 &\frac{2}{|K|}
  [(y_{4}-y_{6})w_{6}]^{2}\\
 &\gtrsim
 &\frac{1}{|K|}
  \max\{|y_{5}-y_{4}|,|y_{4}-y_{6}|,|y_{6}-y_{5}|\}^{2}w_{6}^{2}.\\
 \end{eqnarray*}

 Combining with (\ref{norm2}), we have

 $$\left\|\frac{\partial (\theta w)}{\partial x}\right\|_{0,K}^{2}\gtrsim \frac{1}{|K|}
  \{[(y_{5}-y_{4})w_{4}]^{2}+[(y_{5}-y_{4})w_{5}]^{2}\}.$$

 So
 $$\left\|\frac{\partial (\theta w)}{\partial x}\right\|_{0,K}^{2}\gtrsim \frac{1}{|K|}
  \{\max\{|y_{5}-y_{4}|,|y_{4}-y_{6}|,|y_{6}-y_{5}|\}^{2}w_{6}^{2}+(y_{5}-y_{4})^{2}w_{4}^{2}+(y_{5}-y_{4})^{2}w_{5}^{2}\}.$$
 Notice again, here $E_{min}=\{(x_4,y_4),(x_5,y_5)\}$ and $w_6=w_{E_{min}}^{K_{4}}$, then we get the result.
\end{proof}

 We need to notice that till now we only require the mesh is
 quasi-uniform, since when $\theta$ is convex in element $K_4$,
 the semi-norm of interpolation function $\|\frac{\partial(I_h(\theta w))}{\partial x}\|_{0,K}$
 can be bounded by $\|\frac{\partial (\theta w)}{\partial x}\|_{0,K}$.
 However, this is not true when $\theta(y)$ is concave. In this case,
 $\|\frac{\partial(I_h(\theta w))}{\partial x}\|_{0,K}$ may also depend on some neighboring
 element. To get the information of the neighboring
 element, we assume the partition $\Tri{h}$ is uniform in the following.

 \begin{lemma}\label{lemma: 1 point stability}
 Assume that $\theta$ is nonzero at only one vertex of $K_4$
 and that $K^\prime$ is the unique element from $\Tri{H}$ which has this vertex
on one of its edges. Assume also that $\theta$ is concave in
 $K_4$ (Case 2). Then the following inequality holds
  $$\left\|\frac{\partial (I_{h}(\theta w))}{\partial
 x}\right\|_{0,K}^{2}\lesssim
 \left\|\frac{\partial (\theta w)}{\partial x}\right\|_{0,K'}^{2}.
$$
\end{lemma}

 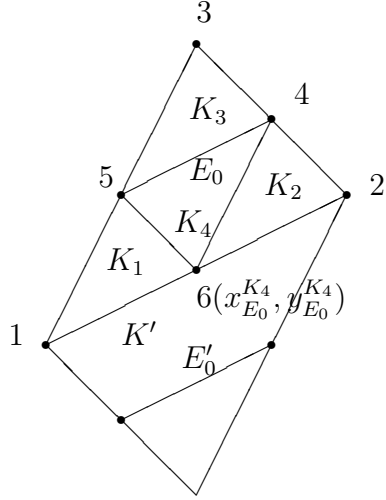
\begin{figure}[!h]
 \begin{center}
 \setlength{\unitlength}{1.0cm}
 \begin{picture}(5,5)
 \put(0,0){\line(1,2){2}}
 \put(0,0){\line(2,1){4}}
 \put(4,2){\line(-1,1){2}}
 \put(1,2){\line(2,1){2}}
 \put(1,2){\line(1,-1){1}}
 \put(2,1){\line(1,2){1}}

 \put(0,0){\circle*{0.1}}\put(-0.5,0){1}
 \put(4,2){\circle*{0.1}}\put(4.3,2){2}
 \put(2,4){\circle*{0.1}}\put(2,4.3){3}
 \put(3,3){\circle*{0.1}}\put(3.3,3.2){4}
 \put(1,2){\circle*{0.1}}\put(0.7,2.1){5}
 \put(2,1){\circle*{0.1}}\put(2,0.5){6$(x_{E_{0}}^{K_4},y_{E_{0}}^{K_4})$}
 %\put(2.3,0.7){$(x_{E_{0}}^{K_4},y_{E_{0}}^{K_4})$}

 \put(0.8,1){$K_{1}$}
 \put(2.9,2){$K_{2}$}
 \put(1.9,3){$K_{3}$}
 \put(1.7,1.5){$K_{4}$}

 \put(2,-2){\line(-1,1){2}}
 \put(2,-2){\line(1,2){2}}
 \put(1,0){$K'$}

 \put(1.9,2.2){$E_0$}
 \put(1,-1){\line(2,1){2}}
 \put(1.8,-0.3){$E_0'$}
 \put(1,-1){\circle*{0.1}}
 \put(3,0){\circle*{0.1}}

 \end{picture}
 \end{center}
 \vspace{2cm}
 \caption{The coarse elements $K$ and $K'$ sharing the same midpoint
 \label{fig:1 point fig}}
 \end{figure}

 \begin{proof}
 Without loss of generality, assume $\theta_{E_{0}}^{K_{4}}$ is the only nonzero
 value. There are two possibilities: (a) $E_{0}=E_{min}$; and  (b) $E_{0}\neq E_{min}$.

\medskip
\textbf{Proof in case (a).} Since $E_{0}=E_{min}$, we conclude that $|\delta_{E_{0}}^{K_{4}} y|   =\min\limits_{E\in\partial
K_{4}}\{|\delta_{E}^{K_{4}} y|\}$.

We  then have
\begin{eqnarray*}
 \left\|\frac{\partial (I_{h}(\theta w))}{\partial x}\right\|_{0,K}^{2}
 &= &\frac{|K|}{4}\sum\limits_{E\in\partial K_{4}} |\delta_{E}^{K_{4}} y|^{2}(\theta_{E_{0}}^{K_{4}}w_{E_{0}}^{K_{4}})^{2}
 \quad (\mbox{from~} \eqref{eq:derivative identity})\\
 &\lesssim &|K|\max\limits_{E\in\partial
  K_{4}}\{|\delta_{E}^{K_{4}}y|\}^{2}(w_{E_{0}}^{K_{4}})^{2}.
 \end{eqnarray*}

 Since the
 partition $\Tri{H}$ is uniform (see Figure \ref{fig:1 point fig}),
 and $K'$ is the element sharing the same point
$(x_{E_{0}}^{K_4},y_{E_{0}}^{K_4})$ with $K$, we know that the  values of $\theta$ at midpoints of $K'$ are all nonzero. This is so, because the
support of $\theta$, whose width is $2h$ must include $K'$ in its interior.
 Assume now that $E'_{0}$ is the edge opposite to point $(x_{E_{0}}^{K_4},y_{E_{0}}^{K_4})$ in $K'_{4}$
 (i.e.  $(x_{E_{0}}^{K_4},y_{E_{0}}^{K_4})=
 (x_{E'_{0}}^{K'_{4}},y_{E'_{0}}^{K'_{4}})$, see Figure~\ref{fig:1
   point fig}). Observe that
 $E'_{0}  =E'_{min}$ or $|\delta_{E'_{0}}^{K'_{4}} y|  =\min\limits_{E'\in\partial K'_{4}}\{|\delta_{E'}^{K'_{4}} y|\}$,
 because $E_0'$ is a parallel translation of $E_0$.
By Proposition \ref{proposition: 3 point property}, we now have
 \begin{equation*}
 \left\|\frac{\partial (\theta w)}{\partial x}\right\|_{0,K'}^{2}
 \gtrsim
  |K'|\max\limits_{E'\in\partial
  K'_{4}}\{|\delta_{E'}^{K'_{4}}y|\}^{2}(w_{E'_{0}}^{K'_{4}})^{2}
 = |K|\max\limits_{E\in\partial
  K_{4}}\{|\delta_{E}^{K_{4}}y|\}^{2}(w_{E_{0}}^{K_{4}})^{2}.\\
 \end{equation*}

 So $\|\frac{\partial (I_{h}(\theta w))}{\partial x}\|_{0,K}^{2}
 \lesssim \|\frac{\partial (\theta w)}{\partial x}\|_{0,K'}^{2}$, and
 this completes the proof in case (a).

\medskip
\textbf{Proof in case (b).}  In case (b) we have $E_{0}\neq E_{min}$ and hence $|\delta_{E_{0}}^{K_{4}} y|
  \neq\min\limits_{E\in\partial K_{4}}\{|\delta_{E}^{K_{4}} y|\}$.  Since $\theta_{E_{0}}^{K_{4}}$ is the only nonzero value among the  values of $\theta$ at the vertices of $K_4$, we easily get
 $$\theta_{E_{0}}^{K_{4}}\lesssim h^{-1}\min\limits_{E\in\partial K_{4}}\{2|K_4|\delta_{E}^{K_{4}} y|\}.$$

 Then
\begin{eqnarray*}
 \left\|\frac{\partial (I_{h}(\theta w))}{\partial x}\right\|_{0,K}^{2}
 &=&\frac{|K|}{4}\sum\limits_{E\in\partial K_{4}} |\delta_{E}^{K_{4}} y|^{2}(\theta_{E_{0}}^{K_{4}}w_{E_{0}}^{K_{4}})^{2}
 \quad (\mbox{from~} \eqref{eq:derivative identity})\\
 &\lesssim& |K|\max\limits_{E\in\partial K_{4}} \{|\delta_{E}^{K_{4}} y|\}^{2}(\theta_{E_{0}}^{K_{4}})^{2}(w_{E_{0}}^{K_{4}})^{2}\\
 &\lesssim&|K|\max\limits_{E\in\partial K_{4}} \{2|K_4|\delta_{E}^{K_{4}} y|\}^{2}
            h^{-2}\min\limits_{E\in\partial K_{4}}\{|\delta_{E}^{K_{4}}
            y|\}^{2}(w_{E_{0}}^{K_{4}})^{2}\\
 &\lesssim& |K|\min\limits_{E\in\partial K_{4}}\{|\delta_{E}^{K_{4}}
            y|\}^{2}(w_{E_{0}}^{K_{4}})^{2}.
 \end{eqnarray*}

 Let $K'$ be the same as before, then by Proposition \ref{proposition: 3 point property}, we have
 \begin{equation*}
 \left\|\frac{\partial (\theta w)}{\partial x}\right\|_{0,K'}^{2}
 \gtrsim |K'|\min\limits_{E'\in\partial
  K'_{4}}\{|\delta_{E'}^{K'_{4}}y|\}^{2}(w_{E'_{0}}^{K'_{4}})^{2}
 = |K|\min\limits_{E\in\partial
  K_{4}}\{|\delta_{E}^{K_{4}}y|\}^{2}(w_{E_{0}}^{K_{4}})^{2}.
 \end{equation*}

 Combining the last two inequalities then gives $\|\frac{\partial
   (I_{h}(\theta w))}{\partial x}\|_{0,K}^{2} \lesssim
 \|\frac{\partial (\theta w)}{\partial x}\|_{0,K'}^{2}$. This
 completes the proof in case (b), and also the proof of the Lemma.
\end{proof}

The next Lemma gives the stability estimates in the last case (Case 3) and we refer to Figure~\ref{fig: 2 point fig} for clarifying the
notation.
\begin{lemma}\label{lemma: 2 point stability}
  Assume that $\theta$ is nonzero at exactly two vertices of $K_4$ and
  concave in $K_4$ (Case 3). Let $K^\prime$ be an element
  from $\Tri{H}$ which shares with $K_4$ the vertex at which $\theta$
  has larger value on one of its edges.  Then the following inequality
  holds
$$\left\|\frac{\partial (I_{h}(\theta w))}{\partial x}\right\|_{0,K}^{2}\lesssim
 \left\|\frac{\partial (\theta w)}{\partial x}\right\|_{0,K\bigcup K'}^{2}.$$
\end{lemma}

 \begin{figure}[!h]
 \vspace{1.5cm}
 \begin{center}
 \setlength{\unitlength}{1.0cm}
 \begin{picture}(5,5)

 \put(0,0){\line(1,2){2}}
 \put(0,0){\line(2,1){4}}
 \put(4,2){\line(-1,1){2}}
 \put(1,2){\line(2,1){2}}
 \put(1,2){\line(1,-1){1}}
 \put(2,1){\line(1,2){1}}

 \put(0,0){\circle*{0.1}}\put(-0.5,0){1}
 \put(4,2){\circle*{0.1}}\put(4.3,2){2}
 \put(2,4){\circle*{0.1}}\put(2,4.3){3}
 \put(3,3){\circle*{0.1}}\put(3.3,3.2){4}
 \put(1,2){\circle*{0.1}}\put(0.7,2.1){5}
 \put(2,1){\circle*{0.1}}\put(2,0.5){6}

 \put(0.8,1){$K_{1}$}
 \put(2.9,2){$K_{2}$}
 \put(1.9,3){$K_{3}$}
 \put(1.8,1.8){$K_{4}$}

 \put(4,2){\line(1,2){2}}
 \put(2,4){\line(2,1){4}}
 \put(3.8,3.5){$K'$}

% \put(0,0){\line(-1,1){2}}
% \put(-2,2){\line(2,1){4}}
% \put(-0.1,1.7){$K''$}

 \end{picture}
 \end{center}
 \caption{The coarse elements $K$ and $K'$ sharing the same
   midpoint\label{fig: 2 point fig}}
 \end{figure}
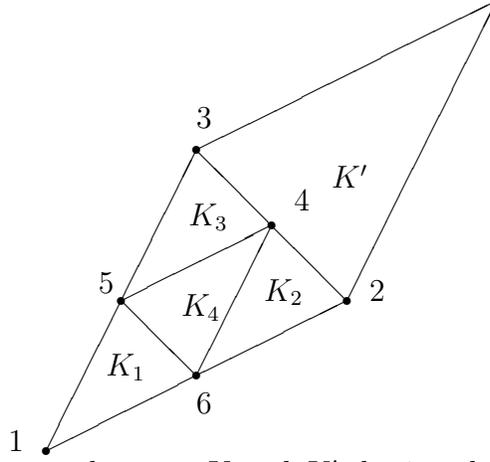

 \begin{proof}
 Without loss of generality we may assume that $\theta_4$ and
 $\theta_5$ are the only nonzero values of $\theta$, we may also
 assume that $\theta_{4}\geq\theta_{5}$. As a consequence,
$K'$ will share $(x_4,y_4)$ with $K$.

We consider two possibilities: (a)  $E_{min}=\{(x_4,y_4),(x_5,y_5)\}$; (b) $E_{min}\neq\{(x_4,y_4),(x_5,y_5)\}$.

\textbf{Proof of (a).} Since $E_{min}=\{(x_4,y_4),(x_5,y_5)\}$, we have that  $|y_{5}-y_{4}|=\min\{|y_{5}-y_{4}|,|y_{4}-y_{6}|,|y_{6}-y_{5}|\}$.
Hence, from~\eqref{eq:derivative identity} we obtain
\begin{eqnarray*}
 \left\|\frac{\partial (I_{h}(\theta w))}{\partial x}\right\|_{0,K}^{2}
 & \thickapprox &
 \frac{1}{|K|}\{
((y_{5}-y_{4})\theta_{4}w_{4})^{2}
 +((y_{5}-y_{4})\theta_{5}w_{5})^{2}\\
&&\,\,\quad +((y_{6}-y_{5})\theta_{5}w_{5}+(y_{4}-y_{6})\theta_{4}w_{4})^{2}\}.
\end{eqnarray*}

We want to bound now all the terms on the right side of the above relation with quantities independent of the values of $\theta$. From the fact
that $\theta \le 1$ we have that
$$((y_{5}-y_{4})\theta_{4}w_{4})^{2}\lesssim ((y_{5}-y_{4})w_{4})^{2}.$$
The other terms are bounded as follows
 \begin{eqnarray*}
 ((y_{5}-y_{4})\theta_5w_{5})^{2}
&\leq &
 ((y_{5}-y_{4})w_{5})^{2} \lesssim
 ((y_{5}-y_{4})w_{5})^{2} (\frac{y_{6}-y_{5}}{y_{4}-y_{6}})^{2}\\
 &= &(((y_{6}-y_{5})w_{5}+(y_{4}-y_{6})w_{4})\frac{y_{5}-y_{4}}{y_{4}-y_{6}}-(y_{5}-y_{4})w_{4})^{2}\\
 &\lesssim
 &((y_{6}-y_{5})w_{5}+(y_{4}-y_{6})w_{4})^{2}(\frac{y_{5}-y_{4}}{y_{4}-y_{6}})^{2}+((y_{5}-y_{4})w_{4})^{2}\\
 &\lesssim
 & ((y_{6}-y_{5})w_{5}+(y_{4}-y_{6})w_{4})^{2}+((y_{5}-y_{4})w_{4})^{2},
 \end{eqnarray*}
and also
 \begin{eqnarray*}
 ((y_{6}-y_{5})\theta_{5}w_{5}+(y_{4}-y_{6})\theta_{4}w_{4})^{2}
 &=&((y_{6}-y_{5})\theta_{5}w_{5}+(y_{4}-y_{6})\theta_{5}w_{4} -(y_{4}-y_{6})(\theta_{5}-\theta_{4})w_{4})^{2}\\
 &=&(((y_{6}-y_{5})w_{5}+(y_{4}-y_{6})w_{4})\theta_{5} -(y_{4}-y_{6})w_{4}\frac{y_{5}-y_{4}}{h})^{2}\\
 &\lesssim&((y_{6}-y_{5})w_{5}+(y_{4}-y_{6})w_{4})^{2}+((y_{5}-y_{4})w_{4})^{2}.
 \end{eqnarray*}

 Hence we get
 \begin{equation*}
 \left\|\frac{\partial (I_{h}(\theta w))}{\partial x}\right\|_{0,K}^{2}
 \lesssim\frac{1}{|K|}\{((y_{5}-y_{4})w_{4})^{2}+((y_{6}-y_{5})w_{5}+(y_{4}-y_{6})w_{4})^{2}\}.\\
 \end{equation*}

 In this case,
 \begin{equation*}
 \left\|\frac{\partial (\theta w)}{\partial x}\right\|_{0,K}^{2}
 \gtrsim \frac{1}{|K|}\{((y_{6}-y_{5})w_{5}+(y_{4}-y_{6})w_{4})^{2}\}.\\
 \end{equation*}

Since $K'$ denotes the element sharing $(x_4,y_4)$ with $K$, we obtain
\begin{eqnarray*}
 \left\|\frac{\partial (\theta w)}{\partial x}\right\|_{0,K'}^{2}
 &\gtrsim &\frac{1}{|K'|}
  \{\min\{|y_{5}^{K'}-y_{4}^{K'}|,|y_{4}^{K'}-y_{6}^{K'}|,|y_{6}^{K'}-y_{5}^{K'}|\}^{2}w_{4}^{2}\}\\
  &=&\frac{1}{|K|}
  \{\min\{|y_{5}-y_{4}|,|y_{4}-y_{6}|,|y_{6}-y_{5}|\}^{2}w_{4}^{2}\}\\
  &=&\frac{1}{|K|}\{(y_{5}-y_{4})^{2}w_{4}^{2}\},
 \end{eqnarray*}
 so
\begin{equation*}
 \left\|\frac{\partial(\theta w)}{\partial x}\right\|_{0,K\bigcup K'}^{2}
 \gtrsim
 \frac{1}{|K|}\{((y_{5}-y_{4})w_{4})^{2}+((y_{6}-y_{5})w_{5}+(y_{4}-y_{6})w_{4})^{2}\}
 \gtrsim \left\|\frac{\partial (I_{h}(\theta w))}{\partial
 x}\right\|_{0,K}^{2}.
 \end{equation*}
This completes the proof in case (a).

\textbf{Proof of (b).} In this case we have that $E_{min}\neq\{(x_4,y_4),(x_5,y_5)\}$, which is equivalent to
$|y_{5}-y_{4}|>\min\{|y_{5}-y_{4}|,|y_{4}-y_{6}|,|y_{6}-y_{5}|\}$.

 Since $\theta_{4}>\theta_{5}$ ($\theta_{4}=\theta_{5}$ can not be true in this case,
  because $\theta_{4}=\theta_{5}$ implies $y_4=y_5$), we can get
  $|y_6-y_5|<|y_4-y_6|$.
 It is then easy to see
 $|y_{6}-y_{5}|=\min\{|y_{5}-y_{4}|,|y_{4}-y_{6}|,|y_{6}-y_{5}|\}$.
 Then
 $$
 \theta_{5}\lesssim h^{-1}
 \min\{|y_{5}-y_{4}|,|y_{4}-y_{6}|,|y_{6}-y_{5}|\}=h^{-1}
 |y_{6}-y_{5}|.
 $$
 So from \eqref{eq:derivative identity}, we have
 \begin{eqnarray*}
 \left\|\frac{\partial (I_{h}(\theta w))}{\partial x}\right\|_{0,K}^{2}
 &\thickapprox
 &\frac{1}{|K|}\{((y_{5}-y_{4})\theta_{5}w_{5})^{2}+((y_{5}-y_{4})\theta_{4}w_{4})^{2}+((y_{6}-y_{5})\theta_{5}w_{5}+(y_{4}-y_{6})\theta_{4}w_{4})^{2}\}\\
 &\lesssim &\frac{1}{|K|}
 \{((y_{5}-y_{4})\theta_{5}w_{5})^{2}+((y_{5}-y_{4})\theta_{4}w_{4})^{2}+((y_{6}-y_{5})\theta_{5}w_{5})^{2}+((y_{4}-y_{6})\theta_{4}w_{4})^{2}\}\\
 &\lesssim &\frac{1}{|K|}
 \{((y_{6}-y_{5})w_{5})^{2}+((y_{4}-y_{6})w_{4})^{2}\}.
 \end{eqnarray*}

 In this case,
 \begin{equation*}
 \left\|\frac{\partial (\theta w)}{\partial x}\right\|_{0,K}^{2}
 \gtrsim \frac{1}{|K|}\{((y_{6}-y_{5})w_{5}+(y_{4}-y_{6})w_{4})^{2}\}.\\
 \end{equation*}

  Since the
 partition $\Tri{H}$ is uniform (see Figure \ref{fig: 2 point fig}),
 and $K'$ is the element sharing the same point
 $(x_4,y_4)$ with $K$, we know that the values of $\theta$ at
midpoints of $K'$ are all nonzero. Observe that the edge opposite to
 point $(x_4,y_4)$ in $K'_4$ is a parallel translation of the edge opposite to
 point $(x_4,y_4)$ in $K_4$. That is to say,
 point $(x_4,y_4)$ in $K'$ (also point 4 in $K$) is just the midpoint opposite to edge with
 $\min\{|y_{5}^{K'}-y_{4}^{K'}|,|y_{4}^{K'}-y_{6}^{K'}|,|y_{6}^{K'}-y_{5}^{K'}|\}$.
 By Proposition \ref{proposition: 3 point property}, we have
 \begin{eqnarray*}
 \left\|\frac{\partial (\theta w)}{\partial x}\right\|_{0,K'}^{2}
 &\gtrsim &\frac{1}{|K'|}
  \{\max\{|y_{5}^{K'}-y_{4}^{K'}|,|y_{4}^{K'}-y_{6}^{K'}|,|y_{6}^{K'}-y_{5}^{K'}|\}^{2}w_{4}^{2}\}\\
 &=&\frac{1}{|K|}
 \{\max\{|y_{5}-y_{4}|,|y_{4}-y_{6}|,|y_{6}-y_{5}|\}^{2}w_{4}^{2}\}.
 \end{eqnarray*}

 Combining the last two inequalities, we have

\begin{eqnarray*}
 \left\|\frac{\partial(\theta w)}{\partial x}\right\|_{0,K\bigcup K'}^{2}
 &\gtrsim
 &\frac{1}{|K|}\{((y_{6}-y_{5})w_{5}+(y_{4}-y_{6})w_{4})^{2}+\max\{|y_{5}-y_{4}|,|y_{4}-y_{6}|,|y_{6}-y_{5}|\}^{2}w_{4}^{2}\}\\
 &\gtrsim
 &\frac{1}{|K|}\{((y_{6}-y_{5})w_{5})^{2}+((y_{4}-y_{6})w_{4})^{2}\}\\
 &\gtrsim &\left\|\frac{\partial (I_{h}(\theta w))}{\partial x}\right\|_{0,K}^{2}.\\
\end{eqnarray*}
 This completes the proof of case (b), and also the proof of Lemma.

\end{proof}

 \begin{lemma}\label{lemma: fine grid}
 For any $w\in V_{h}$, if $I_{H}w=0$, then for any $1\leq i \leq L$,
 $$\left\|\frac{\partial (I_{h}(\theta_{i} w))}{\partial x}\right\|_{0}^{2}\lesssim
 \left\|\frac{\partial (\theta_{i} w)}{\partial x}\right\|_{0}^{2}. $$
 \end{lemma}

 \begin{proof}
 The estimate follows from the local (element-wise) estimates given by
 Lemma \ref{lemma: 3 point stability}, \ref{lemma: 1 point stability}, \ref{lemma: 2 point
 stability}, and summation over all elements from $\Tri{H}$.
 \end{proof}

\section{Proof of the theorem~\ref{theorem: the main theorem}\label{section:Proof}}

In this section we prove the convergence result that we have already stated in Section~\ref{section:Theorem}.

\renewcommand{\thetheorem}{\ref{theorem: the main theorem}}

\begin{theorem}%\label{theorem: the main theorem 2}
For any angle of rotation $\omega\in [0,\pi]$, the two-level iteration with coarse space $V_H$ and line (block) Gauss-Seidel smoother is a
uniformly convergent method. In fact, we have
\begin{equation*}%\label{eq:convergence rate}
 \|E_{TL}\|_{a}^{2}\leq 1-\frac{1}{C},
 \end{equation*}
with constant $C$ independent of $\epsilon$ and $h$.
\end{theorem}
%For any $v\in V_h$ we have
%\begin{equation*}
% \|E_{TL}\|_{a}^{2}\leq 1-\frac{1}{C},
% \end{equation*}
%with constant $C$ independent of $\epsilon$ and $h$.
%\end{theorem}
\addtocounter{theorem}{-1}
\renewcommand{\thetheorem}{\arabic{section}.\arabic{theorem}}

\begin{proof}

From Lemma \ref{lemma: two-level by ludmil}, we only need to prove $\sup\limits_{v\in V_{h}}\frac{\|(I-\Pi_{*})v\|_{*}^{2}}{\|v\|_{a}^{2}}\leq
C$.

 For any $v\in V_h$, let $w:=v-I_{H}v$, $w_{i}:=I_{h}(\theta_{i}(y)w)$, it is easy to see $I_H w=0$, and
\begin{equation*}
\sum_i w_{i}=\sum_i I_{h}(\theta_{i}(y)w)=I_{h}\sum_i(\theta_{i}(y)w)=I_h w=w.
\end{equation*}

 Then
 \begin{eqnarray*}
  \sup_{v\in V_{h}}\frac{\|(I-\Pi_{*})v\|_{*}^{2}}{\|v\|_{a}^{2}}
  &\leq &\sup_{v\in V_{h}}\frac{\|(I-I_H)v\|_{*}^{2}}{\|v\|_{a}^{2}}\\
  &=& \sup_{v\in V_{h}}\frac{\|w\|_{*}^{2}}{\|v\|_{a}^{2}}
  = \sup_{v\in V_{h}}\inf_{\sum_{i}\tilde{w}_{i}=w}\frac{\sum_{i}\|\tilde{w}_{i}\|_{a}^{2}}{\|v\|_{a}^{2}}
  \leq \sup_{v\in V_{h}}
  \frac{\sum_{i}\|w_i\|_{a}^{2}}{\|v\|_{a}^{2}}.
 \end{eqnarray*}

 So we only need to prove for any $v\in V_h$,
 $\sum_{i}\|w_i\|_{a}^{2}\lesssim \|v\|_{a}^{2}$.

 First, the decomposition is stable in $L^{2}$,
 \begin{equation}\label{stable1}
 \sum_{i}\|w_{i}\|_{0}^{2}\lesssim  \|\sum_{i}w_{i}\|_{0}^{2}=\|w\|_{0}^{2}\lesssim
 \sum_{i}\|w_{i}\|_{0}^{2}.
 \end{equation}
 Since the sum is along $x$ direction,
 %\begin{equation}\label{stable2}
% \begin{array}{lll}
\begin{eqnarray}
 \sum\limits_{i}\|\partial_{x}w_{i}\|_{0}^{2}
 &=&\sum\limits_{i}\|\partial_{x}(I_{h}(\theta_{i}(y)w))\|_{0}^{2}\nonumber\\
 &\lesssim  &\sum\limits_{i}\|\partial_{x}(\theta_{i}(y)w)\|_{0}^{2}  \quad(\mbox{by Lemma \ref{lemma: fine grid}})\nonumber\\
 &= &\sum\limits_{i}\|\theta_{i}(y)\partial_{x}w\|_{0}^{2}\label{stable2}\\
 &= &\sum\limits_{i}\|\theta_{i}(y)\|^{2}\|\partial_{x}w\|_{0}^{2}\nonumber\\
 &\lesssim &\|\partial_{x}w\|_{0}^{2},\nonumber
 \end{eqnarray}
% \end{array}
% \end{equation}
 then
 %\begin{equation*}
% \begin{array}{lll}
 \begin{eqnarray*}
 \sum\limits_{i}\|w_{i}\|_a^2
 &= &\sum\limits_{i}\|\partial_{x}w_{i}\|_{0}^{2}+\sum\limits_{i}\epsilon\|\partial_{y}w_{i}\|_{0}^{2}\\
 &\lesssim &\sum\limits_{i}\|\partial_{x}w_{i}\|_{0}^{2}+\sum\limits_{i}\epsilon h^{-2}\|w_{i}\|_{0}^{2} \quad(\mbox{by inverse inequality})\\
 &\lesssim &\|\partial_{x}w\|_{0}^{2}+\epsilon h^{-2}\|w\|_{0}^{2} \quad(\mbox{by (\ref{stable1}) and (\ref{stable2})})\\
 &=&\|\partial_{x}(v-I_{H}v)\|_{0}^{2}+\epsilon h^{-2}\|v-I_{H}v\|_{0}^{2}\\
 &\lesssim &\|\partial_{x}v\|_{0}^{2}+\epsilon
 h^{-2}h^{2}|v|_{1}^{2} \quad(\mbox{by Lemma\,} \ref{lemma: two level convergence})\\
 &\lesssim &|v|_{a}^{2}.
 \end{eqnarray*}
% \end{array}
% \end{equation*}
\end{proof}

\section{Numerical Experiments\label{section:Numerical}}

\subsection{Tests for two-level method on a rotated uniform mesh}
We first test the performance of the two-level iterative method and its convergence properties with respect to $\epsilon$ and $h$. We pick as
initial triangulation a $4\times 4$ mesh with a characteristic mesh size $h_0= \frac{1}{2}\sqrt{2}$ as shown in Figure~\ref{fig:three-domains}.
We then apply the two-level method described earlier on sequence of meshes with mesh sizes $h_k=2^{-k}h_0$, $k=1,\ldots,6$.

 The energy norm of the error of two-level method $\|E_{TL}\|_{a}$ is depicted
 in Figure \ref{fig:uniform} show that two-level method  is uniformly
 convergent  w.r.t. $\epsilon$ and $h$, which agree with the
 theoretical results we have proved in the previous sections.
\begin{figure}[!htb]
{\centering
  \subfloat[][$\omega=0$]{
  \scalebox{0.4}[0.4]{\includegraphics{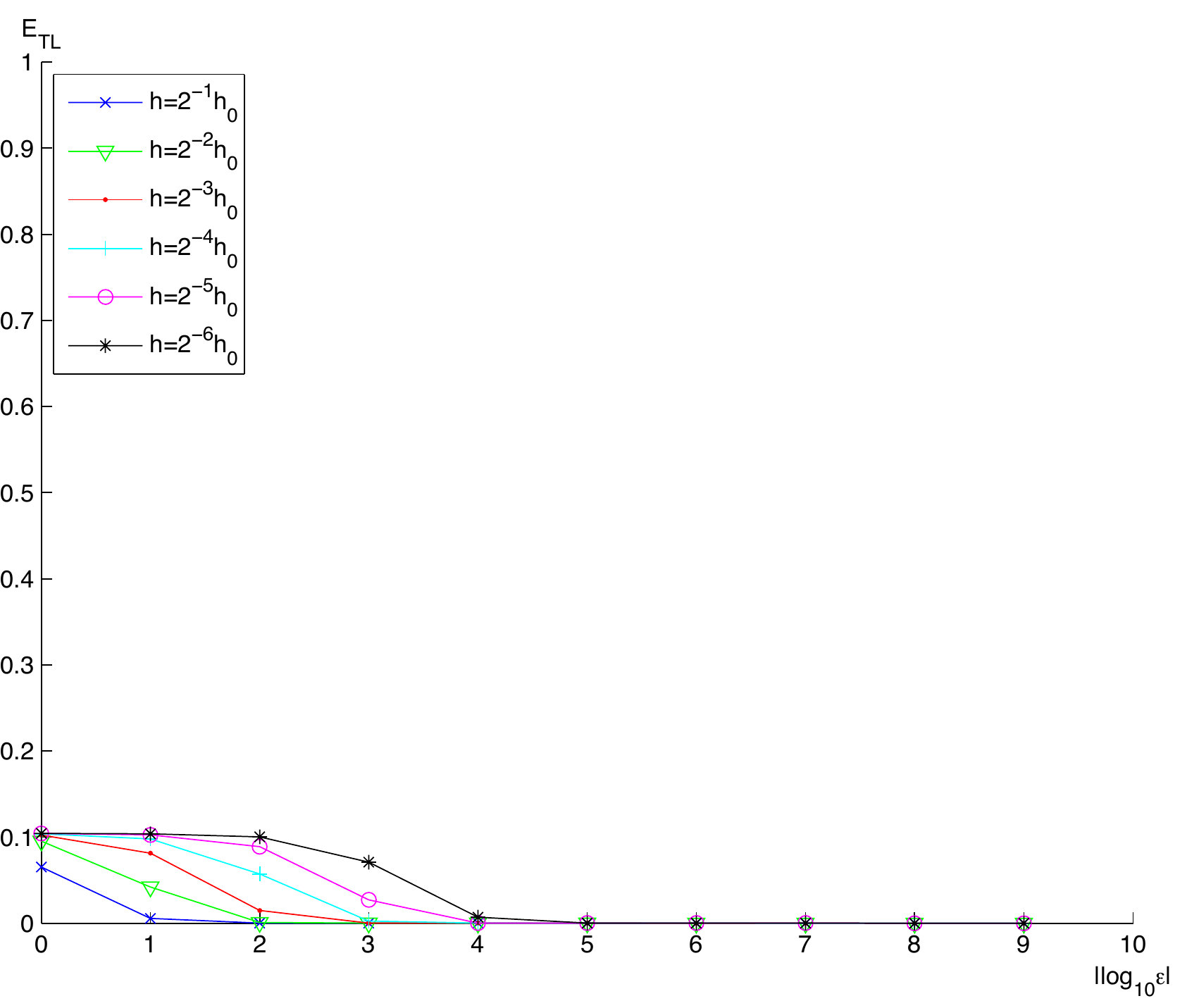}}
}
 \subfloat[][$\omega=\pi/6$]{
     \scalebox{0.4}[0.4]{\includegraphics{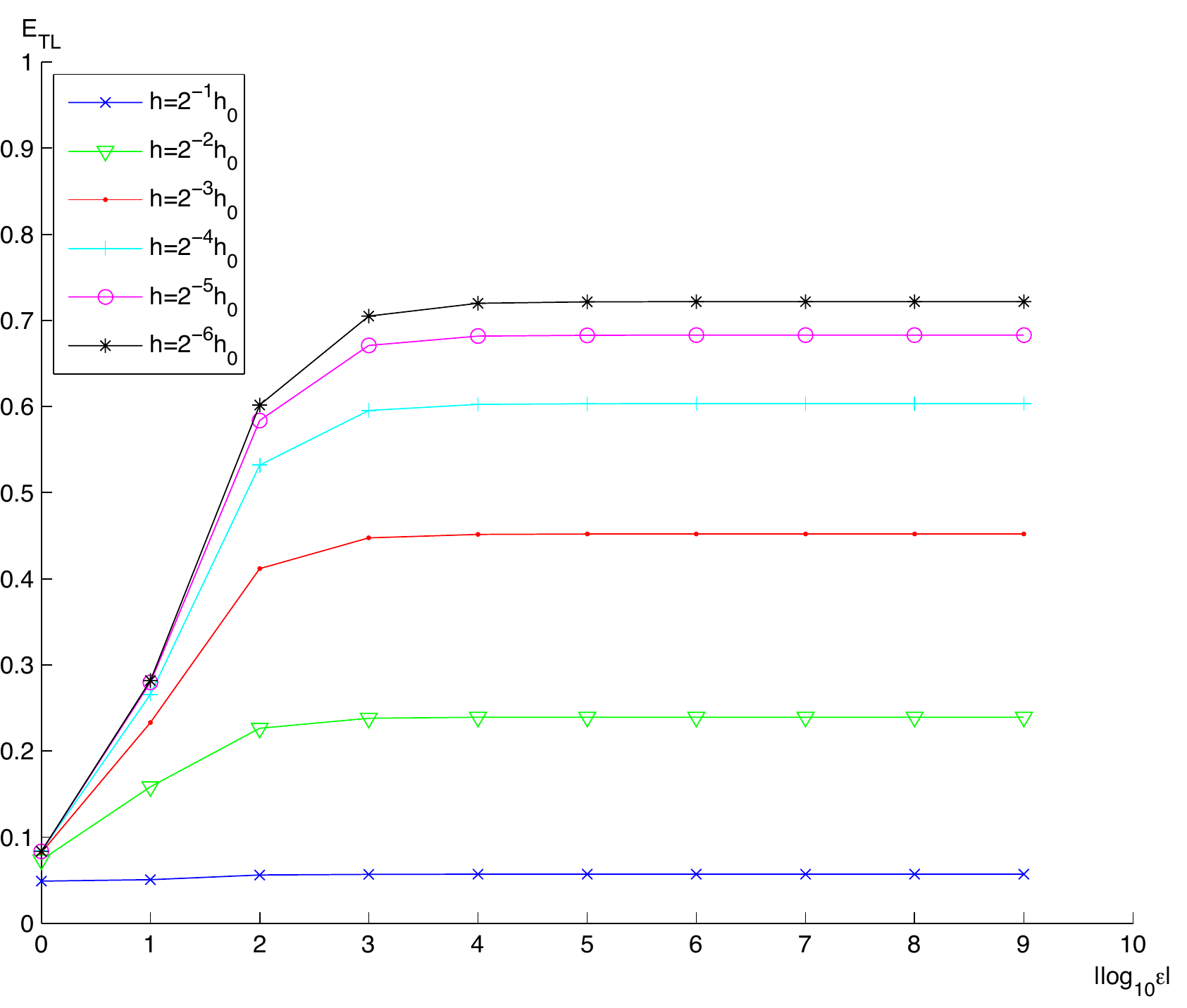}}
} }

{\centering \subfloat[$\omega=\pi/4$]{
     \scalebox{0.4}[0.4]{\includegraphics{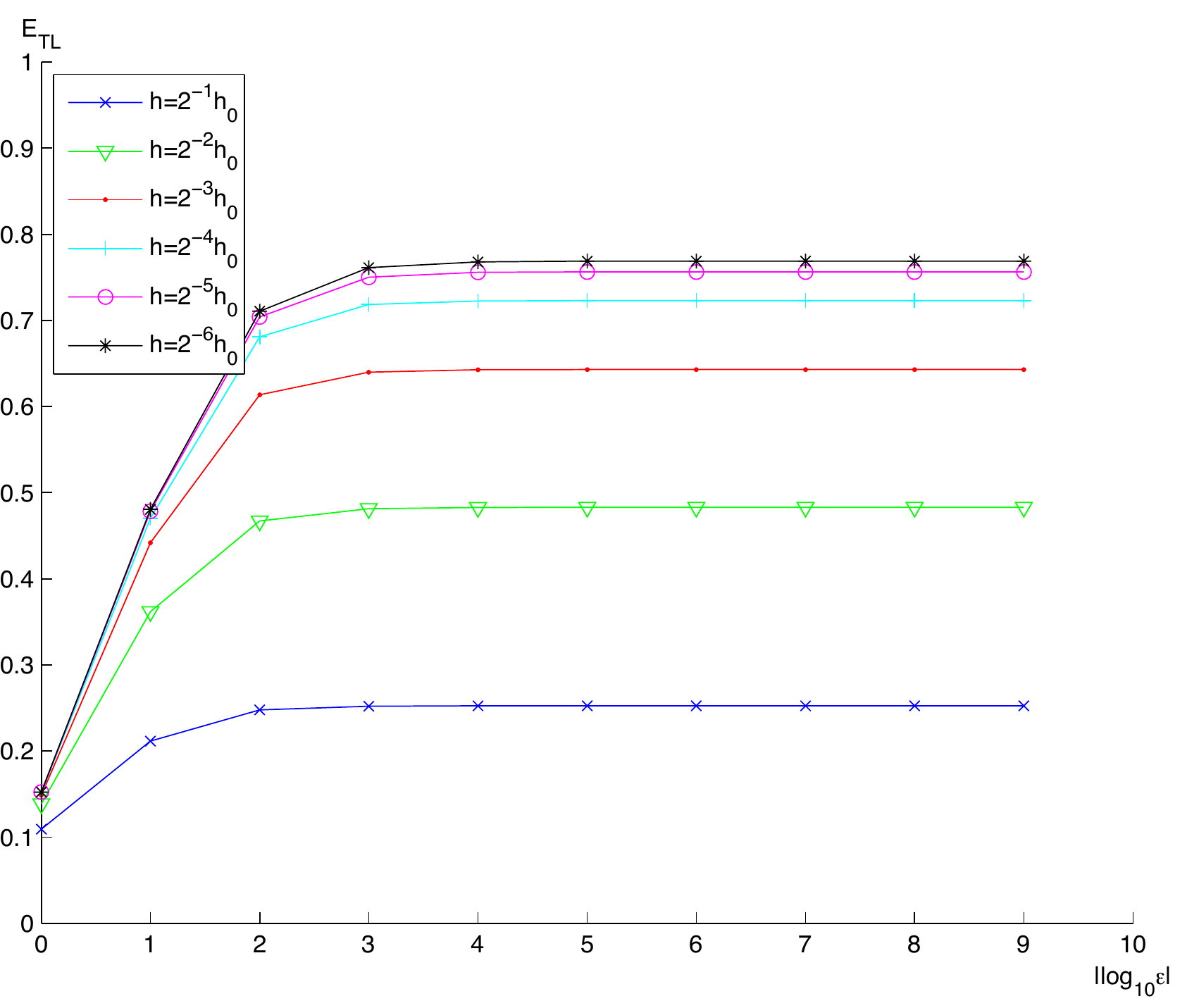}}
} }
 \caption{Plot of the convergence rates $\|E_{TL}\|^2_a$ versus
  $\log_{10}\epsilon$ for a sequence
  of uniform meshes with varying
  angle of anisotropy\label{fig:uniform}
}
\end{figure}

\subsection{Tests for two-level method on a general unstructured mesh}
Similarly to the case of uniform mesh, for a general unstructured mesh we choose $h_0=0.9$ as the maximum diameter of the triangles on the
coarsest mesh $\mathcal{T}_0$ as shown in Figure~\ref{fig:three-general-meshes}. This coarsest mesh is then refined $6$ times and get the mesh
to obtain a sequence of triangulations with characteristic mesh sizes $h=2^{-k}h_0$, $k=1,\ldots,6$.

\begin{figure}[!h]
 \centering
  \scalebox{0.33}[0.33]{\includegraphics{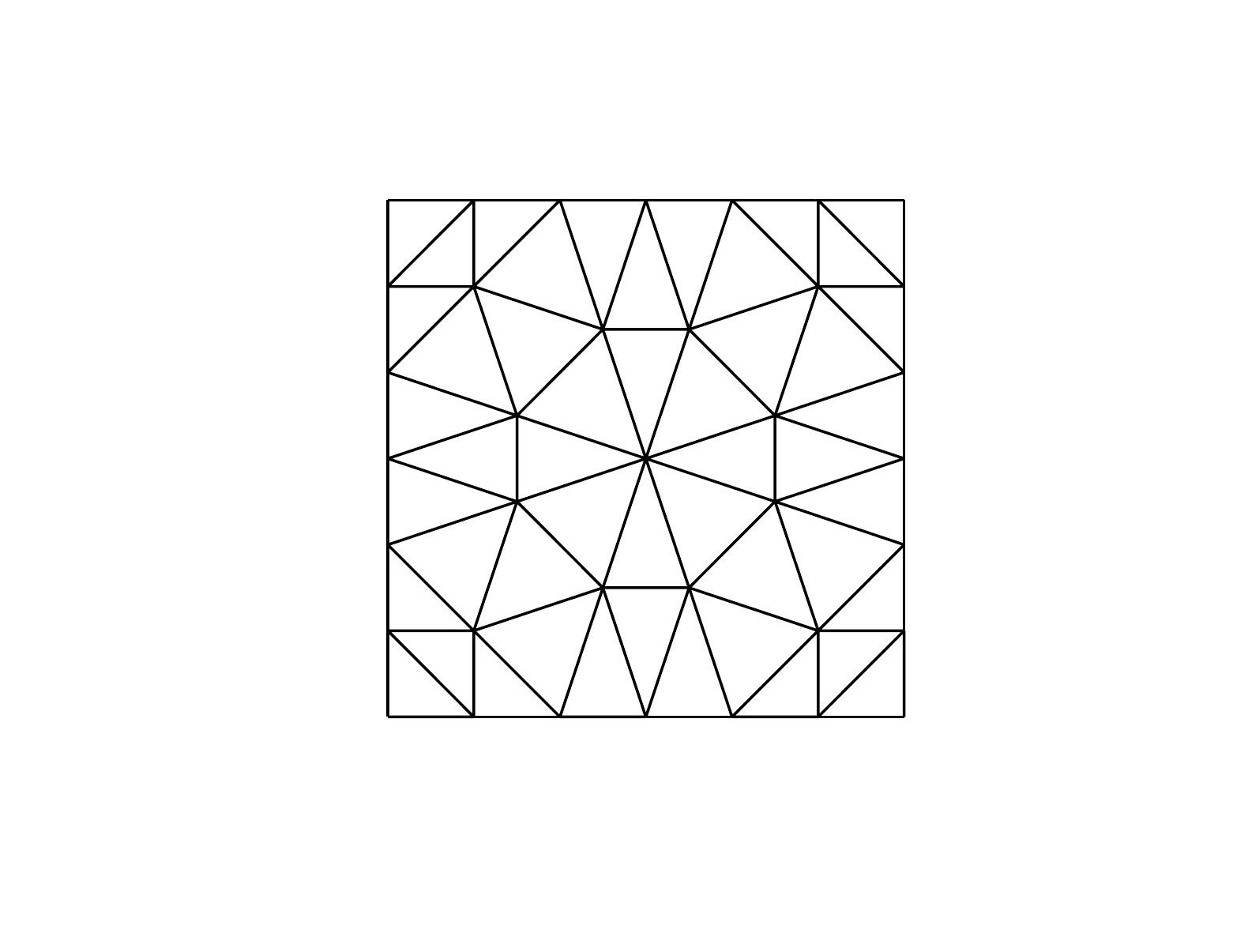}}
 \scalebox{0.33}[0.33]{\includegraphics{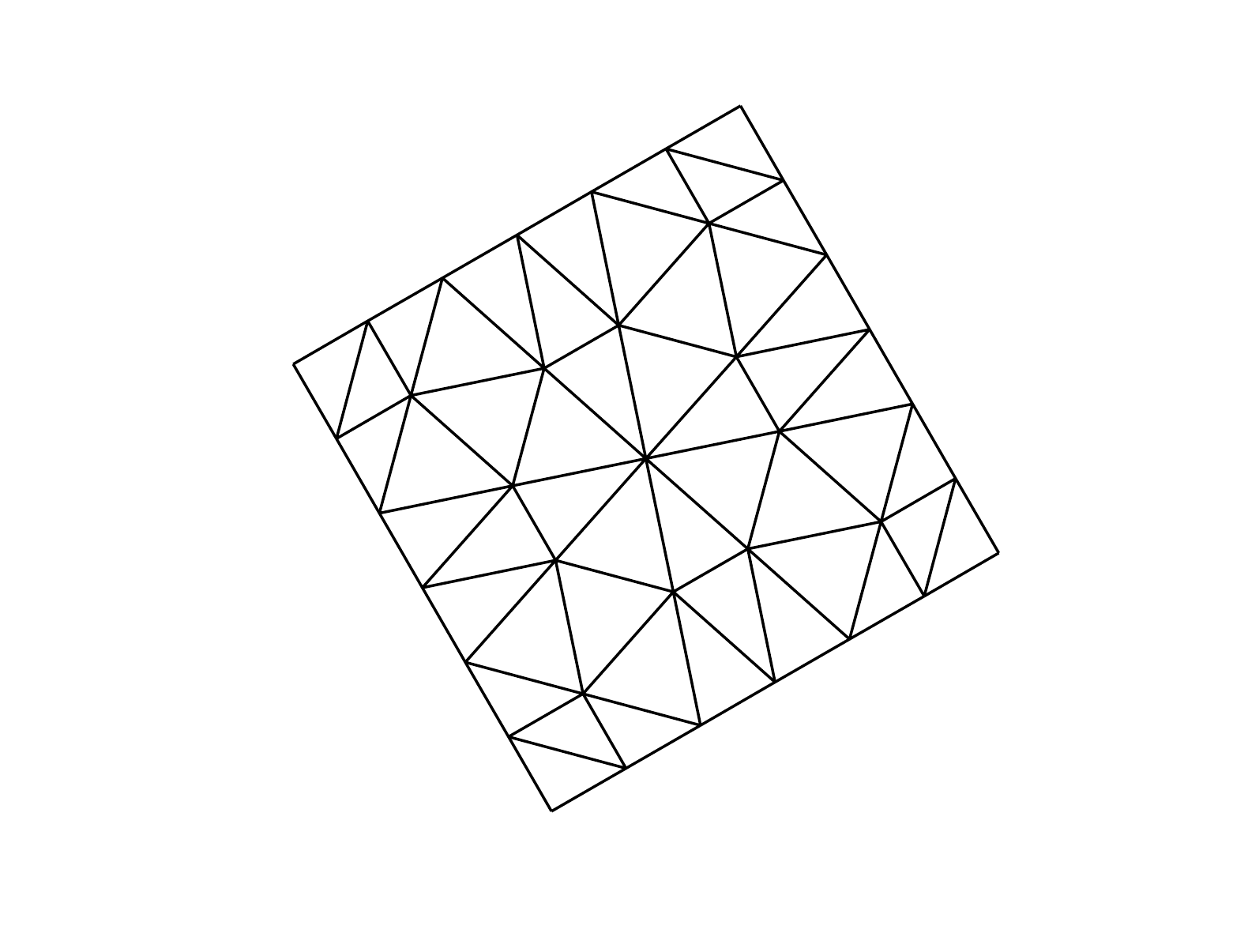}}
 \scalebox{0.33}[0.33]{\includegraphics{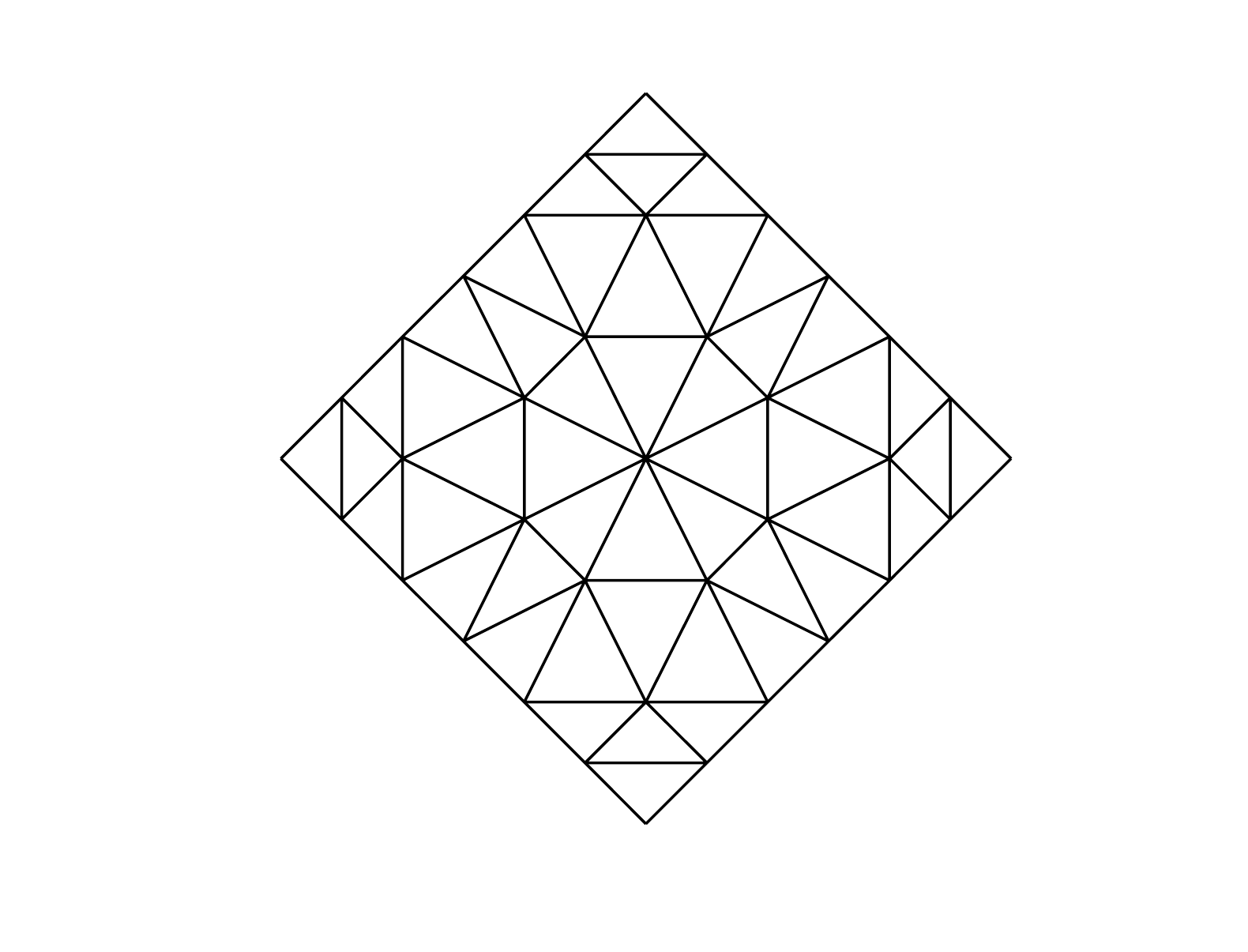}}
 \caption{Plot of unstructured grids used in the numerical examples
    for three values of the angle of
   rotation of anisotropy $\omega=0$, $\omega=\pi/6$ and
   $\omega=\pi/4$.
   \label{fig:three-general-meshes}}
\end{figure}

% \begin{figure}[!h]
% {\centering
%   \subfloat[$\omega=0$]{
%  \scalebox{0.45}[0.45]{\includegraphics{g0fine}}\qquad
% }
%  \subfloat[$\omega=\pi/6$]{
%  \scalebox{0.45}[0.45]{\includegraphics{g30fine}}
% }
% }

% {\centering
% \subfloat[$\omega=\pi/4$]{
% \scalebox{0.45}[0.45]{\includegraphics{g45fine}}
% }
% }
%  \caption{Plot of unstructured grids used in the numerical examples
%     for three values of the angle of
%    rotation of anisotropy $\omega=0$, $\omega=\pi/6$ and
%    $\omega=\pi/4$.
%    \label{fig:three-general-meshes}}
% \end{figure}

 The energy norm of the error propagation operator for the two level
 method, $\|E_{TL}\|_{a}$, are shown
 in Figure~\ref{fig:general}. The uniform convergence is clearly seen
 from the plots. Theoretical justification of such uniform convergence
 is however much more difficult and is a topic of current and future
 research.
\begin{figure}[!h]
{\centering
  \subfloat[$\omega=0$]{
    \scalebox{0.4}[0.4]{\includegraphics{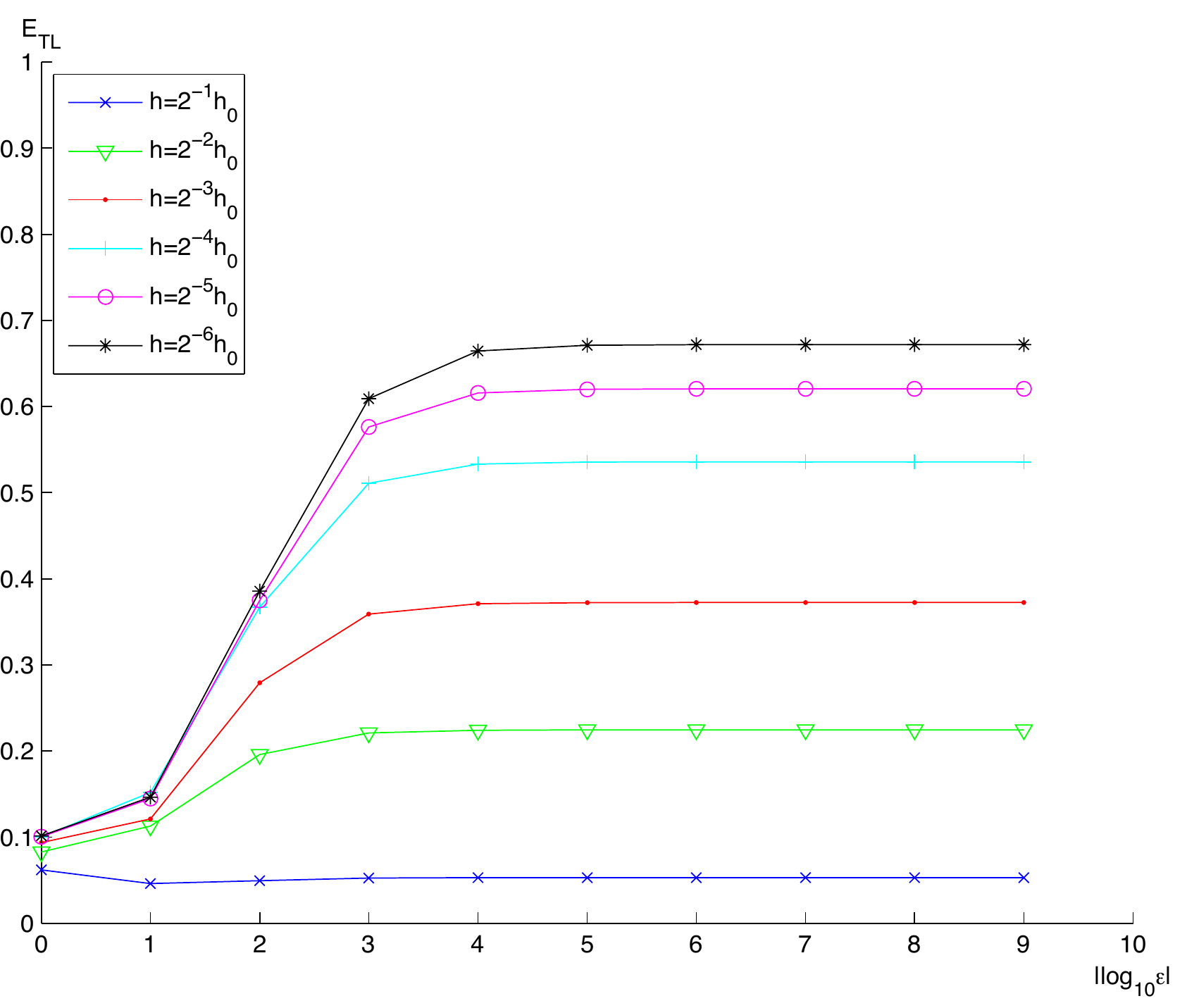}}
}
 \subfloat[$\omega=\pi/6$]{
        \scalebox{0.4}[0.4]{\includegraphics{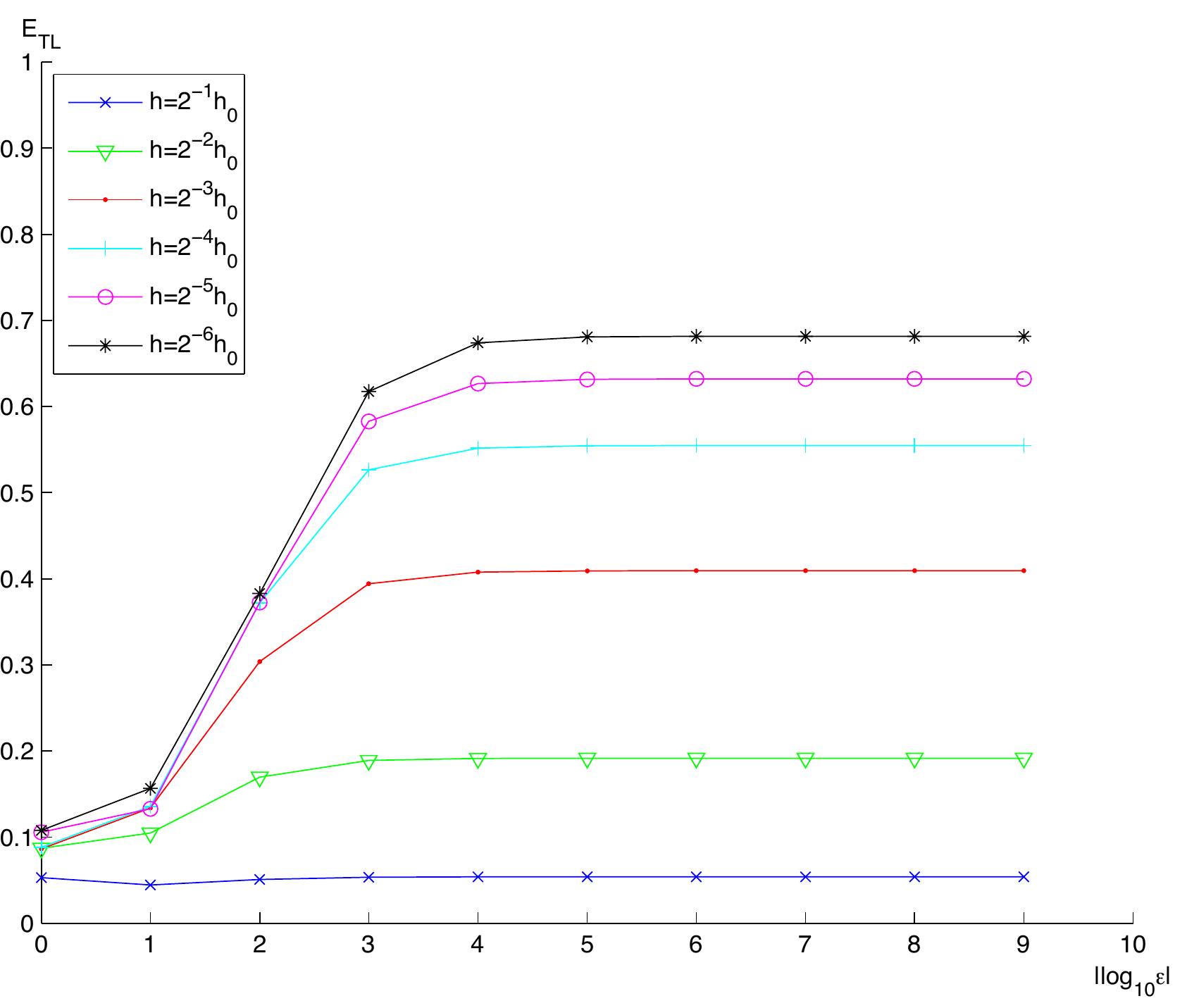}}
} }

{\centering \subfloat[$\omega=\pi/4$]{
        \scalebox{0.4}[0.4]{\includegraphics{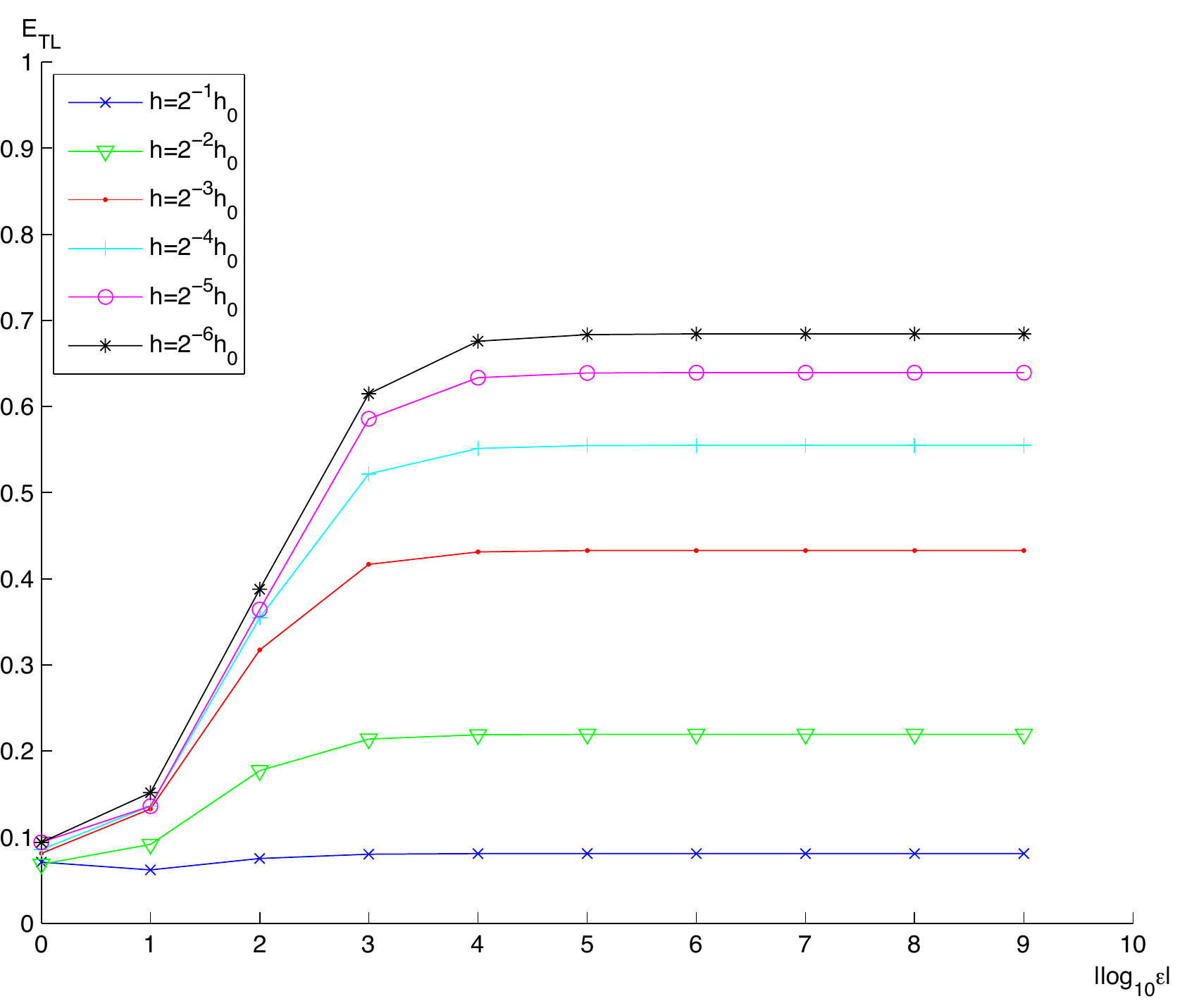}}
} }
 \caption{Plot of the convergence rates $\|E_{TL}\|^2_a$ versus
  $\log_{10}\epsilon$ for a sequence
  of refined general unstructured meshes with varying
  angle of anisotropy\label{fig:general}
}
\end{figure}

\bibliographystyle{unsrt}
\bibliography{GYu_refs}

\end{document}